\documentclass{amsart}

\title{Limits of elliptic hypergeometric biorthogonal functions}
\author{Fokko J. van de Bult and Eric M. Rains}

\usepackage{fullpage}
\usepackage{tikz} %Even better commands for pictures ;-)
\usepackage{multirow}

\newtheorem{theorem}{Theorem}[section]
\newtheorem{definition}[theorem]{Definition}
\newtheorem{proposition}[theorem]{Proposition} 
\newtheorem{corollary}[theorem]{Corollary}
\newtheorem{lemma}[theorem]{Lemma}

\newcommand\T{\rule{0pt}{2.4ex}}   % Command used for table entry after hline (to get proper spacing)
\newcommand\Ttwo{\rule{0pt}{2.8ex}}   % Command used for table entry after hline (to get proper spacing)
\newcommand\B{\rule[-1.2ex]{0pt}{0pt}}   % Command used for table entry before hline

\begin{document}

\begin{abstract}
The purpose of this article is to bring structure to (basic) hypergeometric biorthogonal systems, in particular to the $q$-Askey scheme of basic hypergeometric orthogonal polynomials. We aim to achieve this by looking at the limits as $p\to 0$ of the elliptic hypergeometric biorthogonal functions from \cite{S}, with parameters which depend in varying ways on $p$. As a result we get 38 systems of biorthogonal functions with 
for each system at least one explicit measure for the bilinear form. Amongst these we indeed recover the $q$-Askey scheme. Each system consists of (basic hypergeometric) rational functions or polynomials.
\end{abstract}

\maketitle

Elliptic hypergeometric functions are a generalization of hypergeometric series where the quotient of two subsequent terms is not a rational function of $n$ (ordinary hypergeometric series), or a rational function of $q^n$ (basic hypergeometric series), but an elliptic function of $n$. They were introduced in the late 20th century, and since then many important identities for (basic and ordinary) hypergeometric functions have been generalized to the elliptic level. A nice overview of the study of elliptic hypergeometric functions is given in \cite{Sessay}.

In \cite{SZ1}, \cite{SZ2} and \cite{S}, Spiridonov and Zhedanov studied a family of biorthogonal elliptic hypergeometric functions. These functions can be seen as elliptic versions of the Askey-Wilson polynomials. The functions are defined as a very well poised
elliptic hypergeometric sum ${}_{12}V_{11}$ and the measure corresponds to the integrand of the elliptic beta integral. 

It should be noted that the elliptic hypergeometric functions do not represent polynomials, rather the biorthogonal functions are elliptic functions (of $z$) whose pole locations are restricted. Another significant difference is that on the elliptic level we consider \textit{bi}orthogonal functions, that is, we have two different families of functions, which are orthogonal to each other. The two families of functions are related by interchanging two of their parameters. These two parameters do not occur for Askey-Wilson polynomials and determine the pole locations of the biorthogonal functions. For special values of the parameters a discrete biorthogonality exists, which can be obtained from the continuous measure by residue calculus. 
Just as the continuous orthogonality reduces to that of the Askey-Wilson polynomials, the discrete orthogonality reduces to that of the $q$-Racah polynomials.

The $q$-Askey scheme \cite{KS} gives a description of all known univariate basic hypergeometric families of orthogonal polynomials. It describes the families and the (limit) relations between them, and at the top we have the Askey-Wilson polynomials (i.e. all other families can be written as limits or special cases of the Askey-Wilson polynomials). In \cite{vdBR} the authors have shown that you can obtain a degeneration scheme of basic hypergeometric beta integrals (both the ones with evaluations and those with transformations) by taking the limit as $p\to 0$ from the elliptic hypergeometric beta integral and choosing how the parameters depend on $p$. While it is well known that a proper limit from the elliptic hypergeometric biorthogonal functions yields the Askey-Wilson polynomials, a systematic consideration of what limits can be obtained had not yet been carried out. Inspired by the results of our previous article we set out to discover what limits we could obtain, with the hope that the entire $q$-Askey scheme would appear as limits.

We managed to obtain many different limits of (bi)orthogonal systems of functions, which includes the entire $q$-Askey scheme. This is however just a subpart of our complete scheme. We also obtain numerous pairs of families of rational functions, which are biorthogonal to each other,
families of rational functions biorthogonal to families of polynomials, and families of polynomials biorthogonal on the unit circle (e.g., the Pastro polynomials \cite{Pastro}). 
%For all these limits we have explicit formulas of the functions and measures, and find analogues of the Macdonald conjectures as limits of the elliptic versions. 
Several properties of the elliptic hypergeometric biorthogonal functions should be easily transferable (such as their behavior under difference operators), but we have not yet studied them all.

One particularly pretty way of viewing all these systems and their relations is by considering the polytope $P^{(0)}$ from \cite{vdBR}. In that paper the authors tiled the polytope with three different tiles $P_I^{(0)}$, $P_{II}^{(0)}$ and $P_{III}^{(0)}$. This tiling appears naturally in this paper by considering the Hesse polytope ($P^{(1)}$ in \cite{vdBR}) restricted to a given hyperplane, and subsequently projected to another hyperplane. The tiles then correspond to the images of faces of the restricted Hesse polytope. 

As was shown in \cite{vdBR} each face of any of these tiles corresponds to some limit of the beta integral evaluation, which, as mentioned earlier is the squared norm of the constant function 1. Now we can extend this picture by associating to each simplicial face a system of biorthogonal functions (with orthogonality measure given by the appropriate beta integral limit). One system of biorthogonal functions is a limit of another one if the associated face contains the face of the other system in its boundary. When two faces differ by shifts along vectors in the root lattice of $E_6$, the associated biorthogonal functions are the same, though the two faces give two different measures. 

Let us now give a brief description of which limits we actually considered, and how we discerned in which cases the limits would form a biorthogonal system. We'll let  $\tilde R_{n}(z;t_r;u_r;q;p)$ (where we have 4 parameters $t_r$ and 2 parameters $u_r$) denote the family of elliptic symmetric functions which is biorthogonal to the same functions with the two $u$-parameters interchanged. Now we consider limits as $p\to 0$, while the parameters $z$, $t_r$ and $u_r$ of the biorthogonal function also change in this limit (but $q$ remains fixed). In particular we look at $\tilde R_{n}(z p^{\zeta};t_r p^{\alpha_r};u_rp^{\gamma_r};q;p)$. By a relatively simple argument, these limits always exist, for every choice of $\zeta$, $\alpha_r$ and $\gamma_r$; that is, if we rescale the function by the proper power of $p$ the limit as $p\to 0$ exists and is generically non-zero. However, many of these limits are independent of the variable $z$, which means that they cannot form a family of biorthogonal functions, thus we first filter those out. Our second criterion is to make sure that the limit of the bilinear form specialized to $\tilde R_n$ and $\tilde R_m$, i.e. the expression $\langle \tilde R_{n},\tilde R_{m}\rangle$, still makes sense. Together these conditions determine our classification. It turns out that for each vector satisfying these two conditions we obtain a pair of families of functions which satisfy biorthogonality with respect to a measure we can write down explicitly.

In \cite{vdBRmult} and \cite{vdBRmeas} we show that the same classification also works multivariately, when we take limits of the multivariate $BC_n$-symmetric biorthogonal functions from \cite{Rainstrafo} and \cite{RainsBCn}. Apart from establishing orthogonality in those cases we also obtain generalizations of all other Macdonald conjectures there. The Macdonald polynomials will also fall in our scheme as special cases of multivariate Pastro polynomials.

The article is organized as follows: First we have a notational section, followed by a general section on rings of power series in $p$, which is the space which contains the biorthogonal functions. Subsequently we discuss the definition of the elliptic biorthogonal functions in Section \ref{secbiorthouniv}. Next we discuss the limits of the biorthogonal functions as $p\to 0$ in Section \ref{seclimbiorthouniv}, and in particular answer the question when these limits are $z$-dependent. In Section \ref{seclimsys} we discuss systems of biorthogonal functions (two families of functions and their inner products) and consider when these have proper limits. Subsequently we briefly discuss the types of measures occurring in the limit. 
Section \ref{secAWscheme} compares our results with the $q$-Askey scheme, and the following section gives the relation with Pastro polynomials. Finally, in appendix \ref{secsystemlimit} we explicitly tabulate all limiting systems. 

\section{Notation of univariate $q$-symbols}\label{secnot}
We say a function $f(x;z)$ is written multiplicatively in $x$ if the presence of multiple parameters at the place of $x$ indicates a product; and if $\pm$ symbols in those parameters also indicate a product over all possible combinations of $+$ and $-$ signs. For example
\begin{align*}
f(x_1, x_2, \ldots, x_n;p) &= \prod_{r=1}^n f(x_i;p), \\ f(x^{\pm 1} y^{\pm 1};p) &= 
f(xy;p)f(x/y;p)f(y/x;p)f(1/xy;p).
\end{align*}

Now we define the $q$-symbols and their elliptic analogues as in \cite{GR}. Let $0<|q|,|p|<1$ and set
\begin{align*}
(x;q) &= \prod_{r=0}^\infty (1-xq^r), &
(x;q)_n &= \prod_{r=0}^{n-1} (1-xq^r), & (x;p,q) &= \prod_{r,s\geq 0} (1-xp^rq^s) \\
\theta(x;p) &= (x,p/x;p),&
\theta(x;q;p)_n &= \prod_{r=0}^{n-1} \theta(xq^r;p), & 
\Gamma(x;p,q) & = \prod_{i,j\geq 0} \frac{1-p^{i+1}q^{j+1}/x}{1-p^i q^j x}.
\end{align*}
All these functions are written multiplicatively in $x$. Note that the terminating product $(x;q)_n$ is also defined if $|q|\geq 1$. Likewise $\theta(x;q;p)_n$ is defined for all $q$, though we must still insist on $|p|<1$.

We call a function $f:\mathbb{C}^* \to \mathbb{C}$ symmetric if $f(z)=f(1/z)$. A function is $p$-abelian if it satisfies
$f(z)=f(pz)$. 

We define the space $A(u_0;p,q)$ as the space of all meromorphic symmetric $p$-abelian functions $f$ such that
\[
\theta(pq z^{\pm 1}/u_0;q;p)_{m_f} f(z)
=  \frac{\Gamma(u_0 z^{\pm 1})}{\Gamma(u_0 q^{-m_f} z^{\pm 1})} f(z)
\]
is holomorphic for sufficiently large $m_f$ (where $m_f$ is allowed to depend on $f$). That is, $f$ can only have poles at the points $u_0 q^{-l} p^k$ and
$u_0^{-1} q^l p^k$ for $k\in \mathbb{Z}$ and $1\leq l\leq m_f$, and these poles must be simple.

\section{Function rings}
In this section we define some spaces of functions which contain the functions under consideration. By being slightly more restrictive in the space under consideration than usual we can prove some general results in advance. The variables $x$ in this section will be specialized not just to the variable $z$ of the biorthogonal functions, but also to the parameters $t_r$ and $u_r$.

In principle we can consider the objects we work with as power series in a parameter $p$ (with $|p|<1$). That is
\begin{definition}
We let $F=F(x)$ be the field of formal series $\sum_{t\in T} a_t(x) p^t$, where the set $T$ of exponents is discrete and bounded from below, and the coefficients $a_t(x)$ are (multivariate) rational functions in $x$ independent of $p$, such that $a_{\min T} \neq 0$, together with the function which is constant 0. The valuation of a non-zero element $f=\sum_{t} a_t p^t \in F$ is given by 
$val(f) = \min_{t\in T} t$. The leading coefficient is given by $lc(f) = a_{val(f)}$.
\end{definition} 
If the series $f\in F$ converges (for all small values of $p$), then intuitively the valuation of $f$ describes the asymptotic
size of $f$ as $p\to 0$, while $lc(f)$ describes $\lim_{p \to 0} p^{-val(f)} f$. This means that the goal in this article is to find
leading coefficients (and to be certain we have the leading coefficient we need to find the valuations) of certain specific series (to be defined later). 

One of the properties 
we want to prove is the iterated limit property, Proposition \ref{propiteratedlim}, which is in essence a purely algebraic affair. 
In the statement of this proposition we must be able to substitute $x$ by $p^{\epsilon} x=(p^{\epsilon_1} x_1, \ldots, p^{\epsilon_n}x_n)$. 
To make sense of the new series $f(p^{\epsilon}x)=\sum_{t\in T} a_t(p^{\epsilon} x) p^t$ we must expand each $a_t(p^{\epsilon} x)$ in terms of $p$ and 
then add the resulting series together. This is not possible for every element in $F(x)$; in particular the expansion of 
$a_t(p^{\epsilon} x)$ may have an arbitrarily low valuation, which could give us infinite sums for the individual terms in the expansion of $f(p^{\epsilon} x)$, or violate the boundedness-from-below of the exponent set $T$. This problem can be avoided if we control the complexity of the rational functions $a_t$ in an appropriate way. 

\begin{definition}
We define the degree of a rational function $r\in \mathbb{C}(x_1,\ldots,x_n)$, for any $\epsilon \in \mathbb{R}^n$ as
\[
\deg_\epsilon (r(x)) = val(r(p^{\epsilon}x)), \qquad \deg(r(x)) = -\inf_{\epsilon} \frac{\deg_{\epsilon}(r(x))}{\|\epsilon\|},
\]
where we use the supremum norm for $\|\epsilon\|$.
The subfield $R(x)\subset F(x)$ is given by those series $f=\sum_{t\in T} a_t(x) p^t$ such that
\[
\inf_{t\in T, t>val(f)}  \frac{\deg(a_t)}{t-val(f)} =C < \infty.
\]
\end{definition}
For monomials $x^{u} = \prod_i x_i^{u_i}$, we see that $\deg_{\epsilon}(x^u) = val(x^u p^{\epsilon \cdot u}) = \epsilon\cdot u$. Thus we find that $\deg(x^u) = \sum_{i} |u_i|$, so this is the usual definition of degree for polynomials, and we can now see that our definition of degree corresponds with (total) degree for polynomials. 
For univariate rational functions $r(x) = x^n p(x)/q(x)$ (with $x$, $p$ and $q$ being coprime polynomials and $n\in \mathbb{Z}$), the degree equals $\deg(r) = \max(-n,n+\deg(p)-\deg(q))$ (where we take the usual definition for degree of polynomials). Note that this degree can be negative, for example $\deg(x/(1+x^2))=-1$, even though the degree of all (Laurent) polynomials is positive.
 It should also be observed that $\deg_\epsilon(r(x) \cdot q(x)) = \deg_\epsilon(r(x)) + \deg_{\epsilon}(q(x))$, so $\deg(r(x) \cdot q(x)) \leq \deg(r(x)) + \deg(q(x))$.

To see $R(x)$ is indeed a field, there are two slightly tricky questions. First of all, addition of two elements in $R(x)$ may reduce the valuation (if both leading terms have the same valuation), thereby bringing the denominator of $\deg(a_t)/(t-val(f))$ closer to zero. However the discreteness of the exponent set $T$ ensures that we still obtain a lower bound.

The second tricky part is to show that we are allowed to divide in $R(x)$. We can certainly divide in $F(x)$, thus we 
only need to show the bound on the degrees of the rational functions. Indeed we have for $f=\sum_{t\in T} a_tp^t \in R(x)$
\[
\frac{1}{f}=\frac{ p^{-val(f)}}{lc(f)} \frac{1}{1+\sum_{t\in T, t>val(f)} \frac{a_t}{lc(f)}p^{t-val(f)}}
= \frac{p^{-val(f)}}{lc(f)}  \left(1+ \sum_{n\geq 1} \sum_{\substack{(t_1,\ldots,t_n)\in T^n \\ t_r > val(f)  }} \prod_{k=1}^n \frac{a_{t_k}}{lc(f)} p^{t_k - val(f)} \right),
\]
and we have
\[
\frac{\deg\left( \prod_{k=1}^n \frac{a_{t_k}}{lc(f)} \right)}{\sum_{k=1}^n (t_k-val(f))} \leq \max_{k} \frac{\deg(\frac{a_{t_k}}{lc(f)})}{t_k-val(f)}
\leq C -\frac{\deg(lc(f))}{\min_{t\in T, t>val(f)} t-val(f)}.
\]
In particular the constant $C$ can increase by at most $-\frac{\deg(lc(f))}{\min_{t\in T, t>val(f)} t-val(f)}$.

It can now easily be seen that for arbitrary $f\in R(x)$ and $t\neq val(f)$ we have 
\begin{equation}\label{eqvalbound}
val( a_t(p^{\epsilon}x) p^t) = %-t + val(a_t(p^{\epsilon}x)) =
 t+\deg_{\epsilon}(a_t) \geq t - \|\epsilon\| \deg(a_t) \geq t - \|\epsilon\| (t-val(f)) C.
\end{equation}
In particular if $\|\epsilon\|C < 1$ the valuation becomes bounded from below, and in this case the series
$f(p^{\epsilon}x)$ is well defined. 
Note that we only know that $f(p^{\epsilon}x)\in R(x)$ for some, small enough $\epsilon$. In particular we are unable to consider
$f(p^{\epsilon} x)$ for arbitrary given $\epsilon$.

Now the iterated limit property is almost trivial
\begin{proposition}\label{propiteratedlim}
Let $f\in R(x)$. Then for small enough $\epsilon>0$ and any $u\in \mathbb{R}^n$ we have
\[
lc(lc(f)(p^u x)) = lc(f(p^{\epsilon u}x))
\]
and
\[
val(f)+ \epsilon \  val(lc(f)(p^u x)) = val(f(p^{\epsilon u} x)).
\]
\end{proposition}
\begin{proof}
This is clearly true for series $f\in R$ consisting of a single term $f=lc(f) p^{val(f)}$. If we add higher order terms
$a_t p^t$, they don't change the left hand side, and can only change the right hand side if
$val(a_t(p^{\epsilon u} x)p^t) \geq val(lc(f)(p^{\epsilon u} x)p^{val(f)})$, as otherwise we have a unique lowest order term, determined by the 
leading coefficient. The explicit calculation \eqref{eqvalbound} now shows that if $\epsilon$ is small enough this is indeed the case. Indeed, we can choose $\epsilon$ small enough such that $t + \|\epsilon\| (t-val(f)) C \geq val(lc(f)(p^{\epsilon u} x)p^{val(f)})$ for all $t\in T$ with $t>val(f)$.
\end{proof}
As a corollary we find that we can determine the valuation and leading coefficient of the sum of two terms in some cases where the valuations of the two summands are identical. Of course the case where the two valuations are unequal is trivial.
\begin{corollary}\label{corsumequalval}
Let $f,g\in R(x)$ and define $h=f+g$. Suppose $val(f)=val(g)$, and there exists a $u\in \mathbb{R}^n$ such that 
for all small enough $\epsilon>0$ we have $val(f(p^{\epsilon u}x))< val(g(p^{\epsilon u}x))$. Then $val(h)=val(f)$ and $lc(h)=lc(f)+lc(g)$. 
\end{corollary}
\begin{proof}
First note that as $h(p^{\epsilon u}x) = f(p^{\epsilon u}x)+g(p^{\epsilon u}x)$ is the sum of two terms with different valuations, its valuation is equal to the minimum of those two valuations, i.e. $val(h(p^{\epsilon u}x)) = val(f(p^{\epsilon u}x))$. Now we can calculate using the above proposition 
\[
val(h) + \epsilon \cdot val(lc(h)(p^u x)) = 
val(h(p^{\epsilon u}x)) = val(f(p^{\epsilon u}x)) = 
val(f) + \epsilon \cdot val(lc(f)(p^u x)).
\]
As this should hold for all small enough $\epsilon >0$, we find that
$val(h)=val(f)$ (and $val(lc(h)(p^u x))=val(lc(f)(p^u x))$). As a consequence we see that 
\[
lc(h) = [p^0] h p^{-val(h)} = 
[p^0] fp^{-val(h)} +gp^{-val(h)} =
[p^0] fp^{-val(f)} +gp^{-val(g)} =
lc(f)+lc(g),
\]
where $[p^0] \sum_{t\in T} a_t p^t = a_0$ denotes the coefficient of $p^0$ in the expansion as power series in $p$.
\end{proof}

Of course we are not interested so much in power series, as we are in the (meromorphic) functions they define. 
Note that since we are allowing arbitrary real powers in the series, these
are multi-valued functions of $p$.  However, as in the case of left-bounded
Laurent series, there is a radius $R$ such that for $0<|p|<R$, the series
converges absolutely, and thus converges for any choice of branch.
Thus the space we are really working in will be $M(x)$, defined below.
\begin{definition}\label{defM}
We say a formal power series $f\in F(x)$ converges uniformly on a compact set $K$, if
$p^{-val(f)} f(x)$ converges uniformly for $x\in K$ and $p\in \overline{B_{\epsilon}(0)}=\{z~|~|z|\leq \epsilon\}$ for some $\epsilon>0$. $M(x)$ is the intersection of $R(x)$ with the field of formal power series $f$ which converge uniformly to a holomorphic function
for $x$ in compacta outside the zero-set of some polynomial (which may depend on $f$). 

We also define $\tilde A(u_0;p,q)= A(u_0;p,q) \cap M(u_0)$ as the intersection of these functions with the symmetric $p$-abelian functions defined before.
\end{definition}
Note that the limit of a function $f\in F(x)$ is always its leading coefficient, and thus it is a rational function. This means that if $f$ converges uniformly in a compact set $K$, than $1/f$ converges uniformly on any compact set contained in $K$ minus the zero-set of the numerator of its leading coefficient. In particular if $f\in M(x)$, then so is $1/f$, and therefore $M(x)$ is indeed a field.

Let us now consider some specific elements of $M(x)$. First we define the constant function 
\[
(p;p)_{\infty} = \prod_{j=1}^\infty (1-p^j) = \sum_{n\in \mathbb{Z}} (-1)^n p^{n(3n-1)/2},
\]
which is an element of $M$ where all rational functions (which are the coefficients of $p^t$) are in fact $\pm 1$. Using Jacobi's triple product formula for the theta functions we find
\[
(p;p)_{\infty} \theta(x;p) = \sum_{n\in \mathbb{Z}} (-x)^{n} p^{n(n-1)/2},
\]
so dividing these two elements shows that $\theta(x;p) \in M(x)$ and $1/\theta(x;p)\in M(x)$. In fact we see that
$\theta(p^{\alpha}x;p)\in M(x)$ and $1/\theta(p^{\alpha}x;p) \in M(x)$ for all $\alpha\in \mathbb{R}$, by directly plugging in $x\to p^{\alpha}x$ in the triple product identity. 
Note that for this implication we cannot plug in $p^{\alpha}x$ after viewing $\theta(x;p)$ as an element of $M(x)$, as this could be ill-defined. We can easily determine the valuation and leading coefficient of $\theta(xp^{\alpha};p)$ from this expression and obtain
\[
val(\theta(xp^{\alpha};p)) = \frac12 \{\alpha\}(\{\alpha\}-1)-\frac12 \alpha(\alpha-1) , 
\]
where $\{\alpha\} = \alpha-\lfloor \alpha\rfloor$ denotes the fractional part of $\alpha$, and 
\[
lc(\theta(xp^{\alpha};p)) = \begin{cases} (1-x) (-x)^{-\alpha} & \alpha \in \mathbb{Z}, \\
 (-x)^{-\lfloor \alpha\rfloor} & \alpha \not \in \mathbb{Z}. \end{cases}
\]

We can now conclude that any sum of products and quotients of theta functions $f$, with arguments of the form $p^{\alpha} x^u$ (for some monomial $x^u$) is indeed an element of $M(x)$, and therefore has some valuation $val(f)$ such that $p^{-val(f)} f$ converges uniformly.

\section{Biorthogonal functions}\label{secbiorthouniv}
In this section we recall the definition and basic properties of the biorthogonal functions from \cite{S}. 
%We will be interested in the univariate case, which was already considered by Spiridonov \cite{S}.
\begin{definition}\label{defbiorthouniv}
Let $t_0$, $t_1$, $t_2$, $t_3$, $u_0$, $u_1$, $q$, and $t$ be parameters such that $t_0t_1t_2t_3u_0u_1=pq$. We define

\begin{equation}\label{equnivbiortho}
\tilde R_{n}(z;t_0:t_1,t_2,t_3;u_0,u_1;q;p) := 
\sum_{k=0}^n \frac{\theta(\frac{qt_0}{u_0};q;p)_{2k}}{\theta(\frac{t_0}{u_0};q;p)_{2k}}
\frac{\theta(\frac{t_0}{u_0}, \frac{pq^n}{u_0u_1},q^{-n}, t_0z^{\pm 1}, \frac{q}{u_0t_1}, \frac{q}{u_0t_2}, \frac{q}{u_0t_3};q;p)_k}{\theta(q,\frac{q^{1-n}t_0u_1}{p},\frac{q^{n+1}t_0}{u_0}, \frac{q}{u_0}z^{\pm 1}, t_0t_1,t_0t_2,t_0t_3;q;p)_k} q^{k}
\end{equation}
which is a very well poised elliptic hypergeometric series ${}_{12}V_{11}$.
\end{definition}
The normalization for this definition is chosen so that the biorthogonal functions are highly invariant under shifts of the parameters. 
\begin{lemma}\label{lemRtildeabeluniv}
As functions of $z$ we have $\tilde R_{n} (z;t_0{}:{}t_1,t_2,t_3;u_0,u_1;q;p) \in A(u_0;p,q)$.
Moreover the biorthogonal functions are elliptic in the $t_r$ and $u_r$, that is, they are invariant under multiplying these parameters with integer powers of $p$ (as long as the balancing condition remains satisfied). Finally they satisfy the equations
\[
\tilde R_{n} (zp^{1/2} ;t_0p^{1/2}{}:{}t_1p^{-1/2},t_2p^{-1/2},t_3p^{-1/2};u_0p^{1/2},u_1p^{1/2};q;p) = 
\tilde R_{n} (z ;t_0{}:{}t_1,t_2,t_3;u_0,u_1;q;p)
\]
and
\begin{equation}\label{eqinvquniv}
\tilde R_{n} (z; \frac{1}{t_0}{}:{} \frac{1}{t_1},\frac{1}{t_2},\frac{1}{t_3};\frac{p}{u_0},\frac{p}{u_1};\frac{1}{q};p) = 
\tilde R_{n} (z; t_0 {}:{} t_1,t_2,t_3;u_0,u_1;q;p).
\end{equation}
\end{lemma}
\begin{proof}
The biorthogonal functions are written as sums of functions in the space $A^{(n)}(u_0;p,q)$, so they are in this space themselves. 
The symmetries all follow from $\theta(px;p) = \theta(\frac{1}{x};p) = -\frac{1}{x} \theta(x;p)$. 
%We can use \eqref{eqc0p} and Propositions \ref{propRstarp} and \ref{propbinomp} to show that the individual summands in the definition of the biorthogonal functions satisfy these symmetries. The final equation follows from a direct calculation using \eqref{eqoverq} and \eqref{eqbinomoverq}.
\end{proof}

The biorthogonal functions satisfy a biorthogonality relation. There are two kinds of biorthogonality measures. For
generic parameters we have a continuous biorthogonality, while for specific choices of the parameters the continuous measure
reduces to a discrete one. The discrete version can be obtained from the continuous biorthogonality by residue calculus.
The continuous version was discovered in \cite{S}, while the discrete version was already given in \cite{SZ1} and \cite{SZ2}. While the biorthogonal functions themselves are perfectly well-defined for arbitrary $q$, the continuous measure below only works when $|q|<1$ (otherwise the elliptic gamma functions in the integrand are not well-defined). The discrete measure is a completely algebraic affair, so it works for all values of $q$. 
\begin{theorem}
For any $n,m\in \mathbb{Z}_{\geq 0}$, and for generic values of the parameters such that $t_0t_1t_2t_3u_0u_1=pq$ we have
\begin{multline*}
\langle \tilde R_{n}(\cdot; t_0{}:{}t_1,t_2,t_3;u_0,u_1;q;p), \tilde R_{m}(\cdot ; t_0{}:{}t_1,t_2,t_3;u_1,u_0;q;p) \rangle_{t_0,t_1,t_2,t_3,u_0,u_1;q;p}
\\ =\delta_{n,m} \frac{\theta(\frac{p}{u_0u_1};q;p)_{2n}}{\theta(\frac{pq}{u_0u_1};q;p)_{2n}}
\frac{\theta(q,t_2t_3,t_1t_2,t_1t_3, \frac{qt_0}{u_0},\frac{pqt_0}{u_1};q;p)_n}
{\theta(\frac{p}{u_0u_1},t_0t_1, t_0t_3,t_0t_2, \frac{p}{t_0u_1},\frac{1}{t_0u_0};q;p)_n} q^{-n}
\end{multline*}
where 
\[
\langle f,g\rangle_{t_0,t_1,t_2,t_3,t_4,t_5:q;p} := 
\frac{(q;q) (p;p) }{2 \prod_{0\leq r<s\leq 5} \Gamma(t_rt_s;p,q)} 
\int_{C} f(z) g(z) \frac{\prod_{r=0}^5 \Gamma(t_rz^{\pm 1};p,q)}{\Gamma(z^{\pm 2};p,q)} \frac{dz}{2\pi i z},
\]
for parameters such that $t_0t_1t_2t_3t_4t_5=pq$ and functions $f\in A(t_4;p,q)$ and $g \in A(t_5;p,q)$.
Let $m_f$ be such that $f(z_i) \prod_i \Gamma(t_4z_i^{\pm 1})/\Gamma(t_4q^{-m_f} z_i^{\pm 1})$ is holomorphic, and define $m_g$ likewise for $g$. Let $\tilde t_r=t_r$ for $0\leq r\leq 3$ and $\tilde t_4 = t_4q^{-m_f}$ and $\tilde t_5 =t_5q^{-m_g}$. The contour is now taken to be symmetric ($C=C^{-1}$) and contains all points of the form $p^i q^j \tilde t_r$ (for $i,j\geq 0$) (and hence excludes their reciprocals)\footnote{To be precise, $C$ should be a chain representing the described homology class.}. 

If moreover $t_0t_1=q^{-N}$ (implying no contour of the desired shape exists), and thus $t_2t_3u_0u_1=pq^{N+1}$  we define the inner product as 
\begin{multline*}
\langle f,g\rangle_{t_0,t_1,t_2,t_3,u_0,u_1:q;p}
 :=
\sum_{0\leq k\leq N} 
f(t_0 q^{k}) g(t_0 q^{k}) 
%\frac{\theta(pq t_0^2;q;p)_{2l}}{\theta(pq,t_0^2 q^l;q;p)_l}
%\frac{\theta(t_0t_1,t_0t_2,t_0t_3,t_0u_0,t_0u_1;q;p)_l}{\theta(\frac{pqt_0}{t_1},\frac{pqt_0}{t_2},\frac{pqt_0}{t_3}, \frac{pqt_0}{u_0},\frac{pqt_0}{u_1};q;p)_l}
%\\ \times 
\frac{\theta(q t_0^2;q;p)_{2k}  }{ \theta( t_0^2;q;p)_{2k}   }
\frac{\theta(t_0^2,t_0t_1,t_0t_2,t_0t_3,t_0u_0,\frac{t_0u_1}{p};q;p)_k  }{\theta(q,\frac{qt_0}{t_1},\frac{qt_0}{t_2},\frac{qt_0}{t_3}, \frac{qt_0}{u_0},\frac{pqt_0}{u_1};q;p)_k} q^k
\\ \times 
\frac{\theta( \frac{qt_0}{u_0}, t_1t_2,t_1t_3,\frac{t_1u_1}{p};q;p)_N}{\theta(\frac{t_1}{t_0},\frac{q}{u_0t_2},\frac{q}{u_0t_3},\frac{pq}{u_0u_1};q;p)_N}
\end{multline*}
and have the same biorthogonality, unless $q^k\in p^{\mathbb{Z}}$ for some $0\leq k\leq N$ (in which case one of the point masses becomes infinite).
\end{theorem}
Note that using $n=m=0$ we find that the inner products are normalized such that $\langle 1,1\rangle=1$. The corresponding equations are the famous elliptic beta integral evaluation, respectively the Frenkel-Turaev summation. It should be observed that one over the value of the inner products $\langle \tilde R_l, \tilde R_l\rangle$ gives the summand of a very well poised elliptic hypergeometric series (which, if it terminates can be summed by the Frenkel-Turaev summation). This becomes convenient when studying determinacy of the measure after taking the limit to orthogonal polynomials.

The definition gives us one expansion of $\tilde R_{n}$ as an elliptic hypergeometric series. It is known that the series in question, a ${}_{12}V_{11}$ has a large symmetry group, that is, that there are essentially $72=|W(E_6)|/|S_6|$ ways of writing $\tilde R_{n}$ as such a series. Amongst these symmetries we find that the biorthogonal functions are symmetric up to a constant under permutations of the $t_r$.
The current normalization is such that 
\begin{equation}\label{eqnormRuniv}
\tilde R_{n} (t_0 ;t_0:t_1,t_2,t_3;u_0,u_1) = 1
\end{equation}
as setting $z=t_0$ reduces the series to a single term. 
Together with the permutation symmetry in the $t_r$ this also gives us evaluations for $\tilde R_{n}(t_r)$ for $r=1,2,3$. 

%While not directly relevant to the problem we are considering, we would like to mention the following duality \cite[Theorem 5.4]{RainsBCn}, which is the analog of one of Macdonald's conjectures, and is a generalization of the normalization formula.
%\begin{equation}\label{eqdualityuniv}
%\tilde R_{n}(t_0q^{k};t_0:t_1,t_2,t_3;u_0,u_1) = 
%\tilde R_{k}(\hat t_0q^{n};\hat t_0:\hat t_1,\hat t_2,\hat t_3;\hat u_0,\hat u_1),
%\end{equation}
%where the new parameters are given by 
%\[
%\hat t_0 = \sqrt{t_0t_1t_2t_3/pq}, \qquad 
%\hat t_0\hat t_r =t_0t_r, \quad (r=1,2,3), \qquad 
%\hat t_0/\hat u_r = t_0/u_r, \quad (r=1,2).
%\]

\section{Limits of the biorthogonal functions}\label{seclimbiorthouniv}
The limits that we consider of the biorthogonal functions are the limits as $p\to 0$ in 
\[
\tilde R_n(zp^{\zeta};t_0p^{\alpha_0} : t_1p^{\alpha_1},t_2p^{\alpha_2},t_3p^{\alpha_3};u_0p^{\gamma_0},u_1p^{\gamma_1};q;p),
\]
for different values of $\zeta,\alpha_r,\gamma_r \in \mathbb{R}$, where we assume $t_r$, $u_r$ and $z$ are independent of $p$. That is, we explicitly describe how the parameters $t_r$ and $u_r$ and the variable $z$ depend on $p$. It should be noted that the balancing condition  implies that $\sum_r \alpha_r + \sum_r \gamma_r =1$. Of course the limit will depend on the values of these exponents $\zeta,\alpha_r,\gamma_r$ and we want to determine what the different possible limits are, and when those limits can give us a set of biorthogonal functions. The first necessary condition for obtaining biorthogonality in the limit is that the limiting functions must still be dependent on $z$, so in this section we determine when this is the case.

It should be observed that plugging in these values of the parameters in the definition of the biorthogonal functions, expresses $\tilde R_n$ as a sum of products of theta functions with arguments $p^{\chi}x$ for some $\chi$'s and $x$'s. By the general theory about power series we can thus immediately see that $\tilde R_n \in M(z,t_r,u_r,q)$, and thus that its limit exists (i.e. that it has some valuation and that $\tilde R_n p^{-val(\tilde R_n)}$ converges uniformly to $lc(\tilde R_n)$ as $p\to 0$ outside the zero set of some polynomial). 

What we do not know, however, is if we can plug in these values and then interchange sum and limit in the definition of the biorthogonal functions. Indeed there are many cases in which the valuation of $\tilde R_n$ is much higher than the valuation of the summands. Fortunately we have many different ways of writing the same function (using the symmetries of the ${}_{12}V_{11}$), and
in practical cases we can find a way of writing the function such that this is possible.

For obtaining abstract information about the result of taking this limit, it is more convenient to express the series in terms of an elliptic beta integral. Indeed these series are equal to an $E^1$ from \cite{vdBR} specialized at a point where the integral reduces to a single series of residues. 
In that article the authors studied the limits of $E^1$ as $p\to 0$, so now we can apply the results from there.
%Doing this reduces taking the limit to a known problem. %In particular taking the limit now consists of plugging in known results.
It turns out, we can obtain maximal symmetry in our equations by first using equation \eqref{eqinvquniv} of $\tilde R_{n}$ which inverts $q$. This will impose the somewhat unconventional condition $|q|>1$, but since the results we obtain are about finite series, they are just as valid for arbitrary values of $q$.

Therefore the expression in terms of $E^1$ we will work with is given by 
\begin{multline}
\label{eqrluniv}
 \tilde R_n (z;t_0:t_1,t_2,t_3;u_0,u_1;q,t;p)  %\\  
= 
\frac{E^1(\frac{p^{\frac12} }{q^{\frac12n}t_0}, \frac{p^{\frac12}}{q^{\frac12n}t_1},\frac{p^{\frac12} }{q^{\frac12n}t_2}, \frac{p^{\frac12}}{q^{\frac12n}t_3}, \frac{p^{\frac12}q^{\frac12 n}}{u_0},\frac{p^{\frac12} q^{\frac12 n-1}}{u_1}, p^{\frac12} q^{\frac12 n} z, \frac{q^{\frac12 n}}{ p^{\frac12}  z} ;q^{-1},p)}
{(pq^{-n-1}, \frac{ p t_0z}{q} ,\frac{t_0}{qz}, \frac{p u_0z}{q} ,\frac{u_0}{qz}, \frac{ pu_1z }{q^{n}},\frac{u_1}{q^n z}, \frac{pq^{n-1}}{u_0u_1}, \frac{pq^{n-1}}{t_0u_1} , \frac{p}{t_0u_0};p,q^{-1})}  
 \\  \times 
\frac{(  \frac{t_0t_1}{q}, \frac{t_0t_2}{q},\frac{t_0t_3}{q};p,q^{-1})}
{\prod_{r=1}^3 (\frac{p}{u_0t_r},\frac{p}{qu_1t_r},  \frac{p}{t_0t_r},t_0t_rq^{n-1}, \frac{pt_rz}{q},\frac{t_r}{qz},  \frac{pt_r}{t_1t_2t_3 q^n};p,q^{-1})}  
\frac{(\frac{t_0q^n}{u_0} ,\frac{pu_0}{qt_0}, \frac{t_0u_1}{q^n} ,\frac{1}{u_0z},\frac{pz}{u_0};p,q^{-1})}{(\frac{t_0}{u_0} ,\frac{pu_0}{q^{n+1}t_0} ,t_0u_1,\frac{q^n}{u_0z}, \frac{pq^{n}z}{u_0};p,q^{-1})}
\end{multline}

In this expression we plug in $z\to zp^{\zeta}$, $t_r\to t_rp^{\alpha_r}$ and $u_r \to u_r p^{\gamma_r}$. Subsequently the polytope $P$ in the following proposition gives those vectors for which in \cite{vdBR} the limit of the above univariate biorthogonal function was given.
In this proposition and the rest of this article we use the notation $\alpha_4:=\gamma_0$ and $\alpha_5:=\gamma_1$ to indicate the symmetry between $\alpha_r$'s and $\gamma_r$'s. The symmetries of Lemma \ref{lemRtildeabeluniv} change the vector $(\alpha;\gamma;\zeta)$; For example 
\eqref{eqinvquniv} maps $(\alpha;\gamma;\zeta) \mapsto (-\alpha_0,-\alpha_1,1-\alpha_2,1-\alpha_3;-\gamma_0,-\gamma_1;\zeta)$. The proposition then shows that using these symmetries, any vector $(\alpha;\gamma;\zeta)$ can be mapped into the polytope $P$.  As a consequence we only have to consider the vectors $(\alpha;\gamma;\zeta)\in P$ and thus the limits for any univariate biorthogonal function can be found in \cite{vdBR}.
\begin{proposition}\label{proppolytopes}
The vector $(\frac12-\alpha_0, \frac12-\alpha_1, \ldots, \frac12-\alpha_5, \frac12 +\zeta, -\frac12-\zeta)$ is in the polytope
$P^{(1)}$ (from \cite{vdBR}) if and only if $(\alpha;\zeta)$ is in the polytope $P$ defined by the bounding inequalities
\[
|\frac12 +\zeta|\leq \frac12, \qquad |\frac12 +\zeta|-\frac12 \leq \alpha_i, \qquad
\alpha_i\leq 1+\alpha_j, \qquad 1\geq \alpha_i+\alpha_j, \qquad
\frac32 \geq \alpha_i+\alpha_j+\alpha_k +|\frac12 +\zeta|, \qquad
\sum_i \alpha_i=1.
\]
Let $T$ be the translation group with elements $t_{\alpha}(x) = x+\alpha$, for 
$\alpha$  in the root lattice $\Lambda$ of $E_6$, generated by  the vectors
\begin{multline*}
(-\frac12,-\frac12,-\frac12,\frac12,\frac12,\frac12;\frac12), \quad
(1,-1,0,0,0,0;0), \quad (0,1,-1,0,0,0;0), \\ (0,0,1,-1,0,0;0), \quad (0,0,0,1,-1,0;0), \quad (0,0,0,0,1,-1;0).
\end{multline*}
Let $G$ be the group generated by $T$ and the flip
\[
flip( \alpha;\zeta) = (-\alpha_0,-\alpha_1,1-\alpha_2,1-\alpha_3,-\alpha_4,-\alpha_5 ;\zeta).
\]
Given any vector $(\alpha;\zeta) \in \{v\in \mathbb{R}^6 ~|~ \sum_i v_i = 1\} \times \mathbb{R}$, there exists an element $g\in G$ such that $g(\alpha;\zeta) \in P$.
\end{proposition}
Remark: Under the equation $\sum_i \alpha_i=1$ the relation
$\frac32 \geq \alpha_i+\alpha_j+\alpha_k +|\frac12 +\zeta|$ (for all $0\leq i<j<k\leq 5$) is equivalent to 
$\frac12 +\alpha_i+\alpha_j+\alpha_k \geq |\frac12+\zeta|$ (for all $0\leq i<j<k\leq 5$) and we will use them
both.
\begin{proof}
To obtain the bounding inequalities of the polytope $P$ we just plug the variables $(\frac12-\alpha_0, \frac12-\alpha_1, \ldots, \frac12-\alpha_5, \frac12 +\zeta, -\frac12-\zeta)$ in the bounding inequalities of $P^{(1)}$ and of the resulting equations we scratch those that follow from other equations in that list. The remaining equations are the ones we wrote down.

To show that $P$ contains a fundamental domain for $G$ we first note that, using only the subgroup $T$, any vector $v$ can be mapped to the polytope $B$ defined by the bounding inequalities
\begin{multline*}
\alpha_i-\alpha_j \leq 1 , \quad (1\leq i\neq j\leq 6), \qquad
\alpha_i+\alpha_j+\alpha_k \geq |\zeta+\frac12|-\frac12, \quad (1\leq i<j<k\leq 6), \\
|\frac12 +\zeta|\leq \frac12, \qquad \sum_i \alpha_i =1.
\end{multline*}
Indeed by iterating the operation ``add 1 to the smallest $\alpha_i$ and subtract 1 from the largest $\alpha_i$ if their difference is more than 1'' we can ensure that $\alpha_i-\alpha_j \leq 1$ (this operation must end at some point, as it decreases the non-negative integer $\sum_{i<j} \lfloor |\alpha_i-\alpha_j| \rfloor$). Now we note that $(0;1)$ is in the lattice $\Lambda$, so we can subtract or add an integer to $\zeta$ to ensure $|\zeta+\frac12 |\leq 1/2$. If at this stage $\alpha_i+\alpha_j + \alpha_k \geq |\zeta+\frac12|-\frac12$ for all $i,j,k$, we are done. Otherwise the sum of the three smallest $\alpha_i$ must be less than $|\zeta+\frac12|-\frac12$. Suppose without loss of generality that $\alpha_0 \geq \alpha_1\geq \alpha_2 \geq \alpha_3\geq \alpha_4 \geq \alpha_5$. Now consider the shift of $\alpha$ given by $(\beta;\theta)$, where
$\beta_r = \alpha_r -\frac12$ for $0\leq r\leq 2$ and $\beta_r = \alpha_r+\frac12$ for $3\leq r\leq 5$ and
$\theta=\zeta \pm \frac12$, where we choose the sign to ensure $|\theta+\frac12|\leq \frac12$ (which makes $|\zeta+\frac12|+|\theta+\frac12|=\frac12$). We claim that now $(\beta;\theta) \in B$. Indeed it is clear it still satisfies $\beta_i-\beta_j \leq 1$ (and $|\theta+\frac12|\leq \frac12$). Moreover we see that the smallest three $\beta$'s are exactly the largest $\alpha$'s minus $\frac12$, i.e. $\beta_1$, $\beta_2$ and $\beta_3$. Moreover we have
\[
\beta_0+\beta_1+\beta_2 = \alpha_0+\alpha_1+\alpha_2 -\frac32 = 
1-\alpha_3-\alpha_4-\alpha_5 - \frac32 \geq  -|\zeta+\frac12| = |\frac12+ \theta|-\frac12.
\]

Now the polytope $P$ is a subset of $B$, so we need to use the flip on the set $B \setminus P$ and show that it lands in $P$. This is simplest if we use the ordered versions of $P$ and $B$, that is we add the equations $\alpha_0\leq \alpha_1\leq \cdots \leq \alpha_5$ to the polytopes. We now write 
the complement as the union of two parts, the one part satisfying $|\frac12+\zeta| + \frac12 \leq \alpha_0+\alpha_4+\alpha_5$, and
the other part satisfying $|\frac12+\zeta| + \frac12 \geq \alpha_0+\alpha_4+\alpha_5$. In the first case, in the complement $B\setminus P$, we can assume $\alpha_4+\alpha_5>1$ (as if $\alpha_4+\alpha_5\leq 1$ and we were in the complement we would have $\alpha_0 < |\frac12+\zeta|-\frac12$ contradicting our new equation). Similarly, in the second case we can assume $\alpha_0<|\frac12+\zeta|-\frac12$. 

Now we can look in the first case at the new parameters
$\beta_r=-\alpha_r$ ($0\leq r\leq 3$) and $\beta_r=1-\alpha_r$ ($r=4,5$) (and $\zeta=\zeta$), to land in the polytope with bounding equations
\begin{multline*}
\beta_3\leq \beta_2\leq \beta_1\leq \beta_0\leq \beta_5\leq \beta_4 \leq 1+\beta_3, \qquad
0\leq |\frac12+\zeta|\leq \frac12, \qquad \sum_r \beta_r =1, \\
\beta_4+\beta_5<1, \qquad \frac12\geq \beta_0+\beta_1+\beta_2 + |\zeta+\frac12|, \qquad
\frac32 \geq \beta_0+\beta_4+\beta_5 + |\zeta+\frac12|.
\end{multline*}
This is in the polytope $P$; the only equation which is not immediately clear is $\beta_3\geq |\frac12 +\zeta|-\frac12$, which follows from 
\[
\beta_3 =|\zeta+ \frac12|-(\beta_0+\beta_1+\beta_2+ |\zeta+\frac12|)+ (1-\beta_4-\beta_5) \geq |\zeta+\frac12| - \frac12.
\]

In the second case we consider the new parameters $\beta_0=-\frac12-\alpha_0$ and $\beta_r=\frac12-\alpha_r$ for $1\leq r\leq 5$, and 
$\theta$ such that $|\zeta+\frac12|+|\theta+\frac12|=\frac12$, and see that in that case too we land in $P$. 

We conclude that for any vectors in the complement of $P$ in $B$ we can apply a flip and perhaps some more shifts to get a vector in $P$ itself.
\end{proof}

%Observe that by the proposition we can use the known symmetries (shifts and flip) of the biorthogonal functions to transform any case with vector $(\alpha;\gamma;\zeta)$ to a representation we can take the limit in using the given expression and knowing the limits of $E_1$ from \eqref{eqrluniv} (setting $\alpha_4=\gamma_0$ and $\alpha_5=\gamma_1$).

So now we have to determine which limits are $z$-dependent. We first consider a slightly larger group of biorthogonal functions; those for which either the limit of $E^1$ is $z$-dependent, or the limit of the prefactor is $z$-dependent. 

\begin{proposition}\label{propwhenzdep1univ}
For a vector $(\alpha;\zeta) \in P$, the limit of the $E^1$ in \eqref{eqrluniv} depends on $z$ for precisely one value of $|\zeta+ \frac12|$ given $\alpha$, to wit
$|\zeta+\frac12| = \min(\frac12,\alpha_r+\frac12, \frac32-\alpha_r-\alpha_s-\alpha_t)$ (where $0\leq r<s<t\leq 5$). 
\end{proposition}
\begin{proof}
As shown in \cite{vdBR} the limit of $E^1$ depends only on parameters orthogonal to the plane in which the face lies. 
The polytope in the parameters is a convex set, and for fixed values of $\alpha$ the possible values of $\zeta$ are
given by some interval. Indeed if $m=\min(\frac12,\alpha_r+\frac12, \frac32-\alpha_r-\alpha_s-\alpha_t)$ (for all $r<s<t$) we find that $\zeta \in [-m-\frac12,m-\frac12]$. The interval $(-m-\frac12,m-\frac12)$ is now clearly contained in a single face of the polytope (as all inequalities
with $\zeta$ are strict there), and this face contains the vector in the direction $\zeta$. Hence the resulting limits
are independent of $z$. Thus the only possible value of $|\zeta+\frac12|$ which gives a limit depending on $z$ is $|\zeta+\frac12|=m$.
\end{proof}

The $z$-dependent part of the prefactor can be given as 
\begin{multline*}
\frac{(pz/u_0,1/u_0z ;p,q^{-1})}{(pq^nz/u_0,q^n/u_0z, pu_0z/q, u_0/qz,pu_1z/q^n,u_1/q^nz ;p,q^{-1}) \prod_{r=0}^3 (pt_rz/q,t_r/qz ;p,q^{-1})} \\ =
\prod_{k=0}^{n-1} \frac{1}{(pq^{n-k}z/u_0,q^{n-k} /u_0z;p)}
\frac{1}{(pu_0z/q, u_0/qz,pu_1z/q^n,u_1/q^nz ;p,q^{-1}) \prod_{r=0}^3 (pt_rz/q,t_r/qz ;p,q^{-1})}
\end{multline*}
We can now simply read off the following proposition:
\begin{proposition}\label{propwhenzdep2univ}
Within the polytope $P$ the prefactor above depends in the limit on $z$ only if $\alpha_r+1/2=|\zeta+1/2|$ for some $0\leq r\leq 5$, or $\alpha_4 + |\zeta+\frac12| \geq \frac12$.
\end{proposition}
%
%The limit of this is $z$-dependent only if either $1-\alpha_r \pm \zeta \leq 0$ (for then the $(\frac{pq}{t_r} z^{\pm 1})$ factor would not converge to 1, or for $r=4$ resp. $5$ the factor $(\frac{pq}{u_0} z^{\pm 1})$, resp. $( \frac{q^l}{u_1} z^{\pm 1};p,q)$) or if $\alpha_4 \leq |\zeta|$ (for then the terms $(u_0q^{-k} z^{\pm 1};p,q)$ do not converge to 1). 
%The first case can only happen if the limit of the $E^1$ itself is already $z$-dependent, so it does not give new cases. 
%The second case however does provide some new options (though the actual limits depend in a very precise way on $z$, in fact the only $z$-dependence is through some multiplicative factor.)

Of course, if both the prefactor and the $E_1$ itself are $z$-dependent it is possible that their product is independent of $z$. Thus of the cases above we still need to check whether they are actually $z$-dependent or not. 
\begin{proposition}\label{propzdep}
Within the polytope $P$ the limit of the biorthogonal function is $z$-dependent if and only if the vector $(\alpha;\zeta)$ is in one of the following sets
\begin{itemize}
\item The polytope $P$ intersected with the half space defined by $\alpha_4+|\zeta+\frac12| \geq \frac12$;
\item On the facet of $P$ given by $|\zeta+\frac12| = \frac12 + \alpha_4 + \alpha_r+\alpha_s$ for some $r,s\in \{0,1,2,3,5\}$, $r\neq s$;
\item On the facet of $P$ given by $\alpha_r+\frac12 = |\zeta+\frac12|$ (for $r\in \{0,1,2,3,5\}$) intersected with 
the halfspace $\alpha_4 + |\zeta+\frac12| \leq \frac12$, except in the interior of the resulting polytope;
\item On the facet of $P$ given by $\alpha_4+\frac12 = |\zeta+\frac12|$, except in the interior of this facet;
\end{itemize}
\end{proposition}
\begin{proof}
The limiting functions can only depend on $z$ if either the $E_1$ or the prefactor depends on $z$. This leads us to consider the 
subpolytope of $P$ given by $\alpha_4+|\zeta+\frac12|\geq \frac12$, and the facets where $|\zeta+\frac12| = \frac12$, 
$|\zeta+\frac12| = \alpha_r+\frac12$ (for some $0\leq r\leq 5$), and $|\zeta+\frac12| = \frac32 - \alpha_r-\alpha_s-\alpha_t$ for some distinct $0\leq r,s,t\leq 5$. 

To determine whether the functions in these polytopes are $z$-dependent, it suffices to show that one of their limits is $z$-dependent.
We cut the polytope $P$ according to the hyperplanes given by the equations $|\zeta+\frac12| = \frac12$, $|\zeta+\frac12|=\alpha_r+\frac12$, $|\zeta+\frac12|=\frac32-\alpha_r-\alpha_s-\alpha_t$ and $|\zeta+\frac12|+\alpha_4 = \frac12$. 
The iterated limit property shows that the limit associated to any vector (say in face $F$ of $P$ tiled using these hyperplanes) can be further degenerated to any face which has $F$ as a subface.

For example, take the subpolytope $R$ which is given by the intersection of the half-space  $\alpha_4+|\zeta+\frac12|\geq \frac12$ with $P$. 
For the limit corresponding to any vector in $R$, we can always take a further limit to the limit in the interior of $R$. As the point $(-\frac29, \frac29, \frac29,\frac29;\frac39,\frac29;-\frac14)$ is in the interior, and as at this point only the prefactor depends on $z$, the function associated to the interior of $R$ depends on $z$. Hence all functions associated to any vector in $R$ are $z$-dependent.

This also takes care of the case $|\zeta+\frac12| = \frac12$ and $|\zeta+\frac12| = \frac32-\alpha_r-\alpha_s-\alpha_4$ (with $r\neq s\in \{0,1,2,3,5\}$) as these facets of $P$ are also facets of $R$. If we consider the facet of $P$ given by $|\zeta+\frac12| = \frac12+\alpha_4+\alpha_r+\alpha_s$ (again with $r\neq s\in \{0,1,2,3,5\}$), we note it is cut in two parts by the hyperplane $\alpha_4+|\zeta+\frac12| = \frac12$. On the part $\alpha_4+|\zeta+\frac12| \geq \frac12$ we are in the subpolytope $R$, whereas in the part 
$\alpha_4+|\zeta+\frac12| \leq \frac12$ gives a polytope (of codimension 1) in whose interior only the $E_1$ depends on $z$. By again using the iterated limit argument, all limits associated to vectors on this facet of $P$ depend on $z$.

Let us now consider the case $\alpha_r+\frac12=|\zeta+\frac12|$ for some $r\neq 4$. This facet is again split into  two parts by the hyperplane $\alpha_4+|\zeta+\frac12|=\frac12$, so by the same argument as above, we only need to consider the case for which $\alpha_4+|\zeta+\frac12| \leq \frac12$. In the interior of the subpolytope of $P$ given by these two extra equations, both $E_1$ and the prefactor depend on $z$, and a direct calculation shows that their product in fact becomes constant (i.e. $z$-independent). As noted before, direct calculation of the limit of the $E_1$ involves only plugging the parameters into the results of \cite{vdBR} (in this case using Proposition 4.3 in loc. cit.), and for the prefactor the limits are also immediate. So now we have to consider all facets of this polytope (which are codimension two in $P$), and direct verification shows that the limits associated to these facets all are dependent on $z$. Thus all limits in this polytope are $z$-dependent except the limits associated to the interior. 

The final case is when  $\alpha_4+\frac12 = |\zeta+\frac12|$. In this case we always  have $\alpha_4+|\zeta+\frac12| \leq \frac12$. A direct calculation (using \cite[Prop. 4.3]{vdBR}) again shows that the limit of the biorthogonal functions is constant in the interior of this facet. Thus we calculate the limits of the facets of this facet and observe that those are $z$-dependent. Hence we conclude that the limit is $z$-dependent everywhere on this facet except on the interior.
\end{proof}

Finally obtaining the valuation of $\tilde R_n$ is simple, as we know the valuation of $E_1$ to be zero (in the polytope $P$), and the valuation of the prefactor can be read off immediately.
\begin{proposition}\label{propvaluationbiorthouniv}
For vectors $(\alpha;\zeta)$ in the polytope $P$ the valuation of $\tilde R_n$ is given by 
\begin{multline*}
val(\tilde R_n(zp^{\zeta};t_0p^{\alpha_0}:t_rp^{\alpha_r};u_rp^{\gamma_r})) = 
n \bigg((\alpha_0-\alpha_4) 1_{\{\alpha_0-\alpha_4<0\}}-
(-\zeta-\alpha_4) 1_{\{-\zeta-\alpha_4<0\}} \\ - 
(1+\zeta-\alpha_4) 1_{\{1+\zeta-\alpha_4<0\}}-
(\alpha_0+\alpha_5) 1_{\{\alpha_0+\alpha_5<0\}} -
\sum_{r=1}^3 (\alpha_0+\alpha_r) 1_{\{\alpha_0+\alpha_r<0\}} \bigg)
\end{multline*}
\end{proposition}

It remains to figure out which limits have nice behavior in the inner product.

\section{Limits of biorthogonal systems}\label{seclimsys}
We would like to consider the limits here of pairs of biorthogonal functions, together with their bilinear form. Of course the goal is 
to find new sets of biorthogonal functions, and a corresponding measure. 

For the measure we would like to have a limit equation of the form
\[
lc(\langle f(\cdot;t_rp^{\alpha_r};u_r p^{\gamma_r}),g(\cdot;t_rp^{\alpha_r};u_r p^{\gamma_r}) \rangle_{t_r p^{\alpha_r},u_rp^{\gamma_r}})
= \langle lc( f), lc(g) \rangle_{new, t_r,u_r},
\]
where $lc(f) = lc( f(z p^{\zeta_f};t_rp^{\alpha_r};u_r p^{\gamma_r}))$ and $lc(g) = lc( g(z p^{\zeta_g};t_rp^{\alpha_r};u_r p^{\gamma_r}))$ for some $\zeta_f$ and $\zeta_g$. That is the leading coefficient of the bilinear form should only depend on the leading coefficients of the functions. Since the left hand side of this equation is symmetric in $f$ and $g$, the right hand side must be as well, which implies that we must take $\zeta_f=\zeta_g$, that is, evaluate both functions around the same values of $z$. As the bilinear form is linear in scalars, it is obvious that the valuation of $\langle f,g\rangle$ depends on the valuations of $f$ and $g$. But if those valuations are given the valuation of the inner product should no longer depend on $f$ and $g$ (at least for generic $f$ and $g$). That is, we should have 
\begin{multline*}
val(\langle f(\cdot;t_rp^{\alpha_r};u_r p^{\gamma_r}),g(\cdot;t_rp^{\alpha_r};u_r p^{\gamma_r}) \rangle_{t_r p^{\alpha_r},u_rp^{\gamma_r}})
\\ = val(f(z p^{\zeta_f};t_rp^{\alpha_r};u_r p^{\gamma_r})) + val(g(z p^{\zeta_f};t_rp^{\alpha_r};u_r p^{\gamma_r})) +C(\alpha,\gamma),
\end{multline*}
for some $C(\alpha,\gamma)$. But as we scaled our bilinear form to have $\langle 1, 1\rangle =1$, we see that $C(\alpha,\gamma)=0$, thus we get 
\begin{multline}\label{eqbilval}
val(\langle f(\cdot;t_rp^{\alpha_r};u_r p^{\gamma_r}),g(\cdot;t_rp^{\alpha_r};u_r p^{\gamma_r}) \rangle_{t_r p^{\alpha_r},u_rp^{\gamma_r}})
\\ = val(f(z p^{\zeta_f};t_rp^{\alpha_r};u_r p^{\gamma_r})) + val(g(z p^{\zeta_f};t_rp^{\alpha_r};u_r p^{\gamma_r})).
\end{multline}
It should be observed that we want an equation like this only for generic functions $f$ and $g$. The case where $f$ and $g$ are orthogonal would clearly violate such genericity (and in that case this equation would not make much sense). But expanding arbitrary functions in our biorthogonal families shows that taking $f$ and $g$ both equal to $\tilde R_n$ (with interchanged parameters $u_0$ and $u_1$) is generic enough. 
Thus we can consider the explicit biorthogonality norm relation
\begin{multline*}
\langle \tilde R_{n}(\cdot; t_0{}:{}t_1,t_2,t_3;u_0,u_1;q;p), \tilde R_{n}(\cdot ; t_0{}:{}t_1,t_2,t_3;u_1,u_0;q;p) \rangle_{t_0,t_1,t_2,t_3,u_0,u_1;q;p}
\\ = \frac{\theta(\frac{p}{u_0u_1};q;p)_{2n}}{\theta(\frac{pq}{u_0u_1};q;p)_{2n}}
\frac{\theta(q,t_2t_3,t_1t_2,t_1t_3, \frac{qt_0}{u_0},\frac{pqt_0}{u_1};q;p)_n}
{\theta(\frac{p}{u_0u_1},t_0t_1, t_0t_3,t_0t_2, \frac{p}{t_0u_1},\frac{1}{t_0u_0};q;p)_n} q^{-n}
\end{multline*}
If we choose $\zeta$ (the location where we are evaluating the biorthogonal functions), then the valuation on the left hand side immediately follows from the valuations of the biorthogonal functions we obtained in the previous section. The valuation on the right hand side can easily be obtained. Thus this relation provides us with an equation between the valuations that must be satisfied in order for the system of biorthogonal functions to have a nice limit. The reason this equation might fail is if 
the valuation on the left hand side is not given by \eqref{eqbilval}. 
%it is not true for these functions that the valuations of the bilinear form equals the sum of the valuations of the constituents. 
With the usual rescaling of the bilinear form the limit of 
$\langle \tilde R_{n}, \tilde R_{n} \rangle$ would then vanish. Thus in that case, either the limits of the biorthogonal functions are not linearly independent, or the limit of the bilinear form is not non-degenerate. Note that it can only be the case that the apparent valuation on the left hand side is less than its actual valuation, which equals the valuation on the right hand side, as only on the left hand side some extra cancellation can happen which we have not considered before.

The valuation of the squared norm on the right hand side is a piecewise linear function of the $\alpha_r$, which is $1$-periodic in all $\alpha_r$ (assuming the balancing condition remains satisfied), by ellipticness of the biorthogonal functions (and thus their norms). However it is rather complicated to write down explicitly for all $\alpha_r$. In the previous section we have seen that we could restrict ourselves to the case $\alpha \in P$ (from Proposition \ref{proppolytopes}), so we will do the same here. In the polytope $P$ the norm has valuation (remember $\alpha_4=\gamma_0$ and $\alpha_5=\gamma_1$)
\begin{multline}\label{eqvalnorm}
n\bigg( -(\gamma_0+\gamma_1) 1_{\{\gamma_0+\gamma_1>0\}} -2(\gamma_0+\gamma_1) 1_{\{\gamma_0+\gamma_1<0\}} +
\sum_{1\leq r<s\leq 3} (\alpha_r+\alpha_s)1_{\{\alpha_r+\alpha_s<0\}}
\\
-\sum_{r=1}^3 (\alpha_r+\alpha_0) 1_{\{\alpha_r+\alpha_0<0\}} + \sum_{r=0}^1 \left(
(\gamma_r-\alpha_0)1_{\{\alpha_0-\gamma_r>0\}}+(\alpha_0+\gamma_r)1_{\{\alpha_0+\gamma_r>0\}}  \right) \bigg).
\end{multline}
However, we are not so much interested in this valuation, as we are in this valuation minus the valuations of the biorthogonal functions which are the arguments of the inner product. The valuations of these biorthogonal functions are given in Proposition \ref{propvaluationbiorthouniv}. Hence if we take the difference and divide by $n$, we find that we obtain an interesting limit if
\begin{equation}\label{eqvalipuniv}
(\gamma_0+\gamma_1)1_{\{\gamma_0+\gamma_1>0\}} + \sum_{0\leq r<s\leq 3} (\alpha_r+\alpha_s) 1_{\{\alpha_r+\alpha_s<0\}}
+\sum_{r=0}^1 -(\zeta+\gamma_r) 1_{\{\gamma_r+\zeta>0\}} + (1+\zeta-\gamma_r) 1_{\{1+\zeta<\gamma_r\}} 
\end{equation}
equals zero. If this is positive the inner product will converge to zero (so we won't get anything interesting), and if this term is negative we must have made a mistake (as in that case a perfectly valid limit on the left diverges on the right). Fortunately a (rather tedious) case analysis shows that latter case indeed never happens.

It remains to determine when this term vanishes. One of the surprising results of the proposition below is that while both conditions ($z$-dependence and vanishing of this piecewise linear term) break the symmetry between the $\alpha_r$'s ($0\leq r\leq 3$) and the $\gamma_r$'s, the final result is symmetric.
\begin{proposition}\label{propsystemsuniv}
Consider the polytope $P^{(0)}$ (as in \cite{vdBR}) given by the bounding inequalities
\[
\alpha_r\geq -\frac12, \qquad 
\alpha_r-\alpha_s\leq 1, \qquad 
\alpha_r+\alpha_s\leq 1, \qquad 
\sum_r \alpha_r =1.
\]
Moreover define $P_{II,t}$ ($P_{II,0} = P_{II}^{(0)}$ from \cite{vdBR}) for $0\leq t\leq 5$ as the polytope
\[
-\frac12\leq \alpha_t\leq 0, \qquad 
\alpha_t\leq \alpha_r\leq 1+\alpha_t, \quad (r\neq t), \qquad 
0\leq \alpha_r+\alpha_s\leq 1, \quad (r\neq s \neq t), \qquad
\sum_r \alpha_r=1.
\]

The set of vectors in $P$, for which each part of the associated limiting pair of the biorthogonal functions is $z$-dependent, and for which \eqref{eqvalipuniv} vanishes consists of all vectors in $P^{(0)}$, outside the interiors of the subpolytopes $P_{II,t}$ ($0\leq t\leq 5$), with
the value of $\zeta$ given by  $|\zeta+\frac12|=\min(\frac12,\alpha_r+\frac12,\alpha_r+\alpha_s+\alpha_t+\frac12)$.
\end{proposition}
\begin{proof}
Note that if $\gamma_0+\gamma_1\leq 0$ then the sum \eqref{eqvalipuniv} consists of no terms (and thus equals 0). Indeed we find for $0\leq r<s\leq 3$ that $\alpha_r+\alpha_s=1-\alpha_t-\alpha_u-(\gamma_0+\gamma_1) \geq 0$ (where $t$ and $u$ are such that $\{r,s,t,u\}=\{0,1,2,3\}$). And moreover $\gamma_0 \leq -\gamma_1 \leq \frac12 - |\zeta+\frac12|$ thus 
$\gamma_0 +\zeta\leq 0$ and $\gamma_0 \leq \zeta+1$ and similarly for $\gamma_1$.

Moreover we see that at most three terms in the sum $\sum_{0\leq r<s\leq 3} \cdots$ can be non-zero, as an argument as before gives that either $\alpha_0+\alpha_1 \geq 0$ or $\alpha_2+\alpha_3\geq 0$ (or both). This gives us 5 options up to symmetry: 
\begin{enumerate}
\item $\alpha_r+\alpha_s\geq 0$ for all $0\leq r<s\leq 3$;
\item $\alpha_0+\alpha_1 \leq 0$ and the sum of all other pairs is positive;
\item $\alpha_0+\alpha_1\leq 0$, $\alpha_0+\alpha_2\leq 0$,  and the sum of all other pairs is positive;
\item $\alpha_0+\alpha_r \leq 0$ for $1\leq r\leq 3$ and the sum of all other pairs is positive;
\item $\alpha_r+\alpha_s\leq 0$ for $0\leq r<s\leq 2$ and the sum of all other pairs is positive.
\end{enumerate}
For the $\sum_{r=0}^1 \cdots$ part we also note that we cannot have both $\gamma_0+\zeta>0$ and $1+\zeta<\gamma_1$ as
$\gamma_0+\gamma_1\leq 1$. So there we have at most 2 non-zero terms and, up to symmetry (where we use $\zeta\to 1-\zeta$ symmetry to choose $-\frac12\leq \zeta\leq 0$), 4 options:
\begin{enumerate}
\item[(A)] $\gamma_0,\gamma_1 \leq -\zeta, 1+\zeta$;
\item[(B)] $\gamma_1 \leq -\zeta \leq  \gamma_0 \leq 1+\zeta$;
\item[(C)] $-\zeta \leq  \gamma_0,\gamma_1 \leq 1+\zeta$;
\item[(D)] $\gamma_1 \leq -\zeta  \leq 1+\zeta \leq \gamma_0$
\end{enumerate}
Within each of the $5\times 4=20$ cases described by the above conditions the expression in \eqref{eqvalipuniv} becomes linear, and thus reduces to a simple linear condition (though in case $1A$, i.e. the intersection of cases 1 and $A$, the condition is $0=0$). 

Now we recall from Proposition \ref{propzdep} that if the limit of the biorthogonal functions on both sides of the bilinear form is $z$-dependent and $|\zeta+\frac12| \neq \min(\frac12,\alpha_r+\frac12,\alpha_r+\alpha_s+\alpha_t+\frac12)$ (where we allow $0\leq r<s<t\leq 5$) then we must have  $\frac12-|\zeta+\frac12|\leq \gamma_0, \gamma_1$, where we need both inequalities to ensure both elements of the pair are $z$-dependent. This puts us certainly in case C above. For each of the five different case for the signs of $\alpha_r+\alpha_s$ ($0\leq r<s\leq 3$) the condition that the norm of the  bilinear form has the right valuation gives us an equation for $|\zeta+\frac12|$. 
Inspection learns that in each case this equation is $|\zeta+\frac12|=\min(\frac12,\alpha_r+\frac12,\alpha_r+\alpha_s+\alpha_t+\frac12)$. 
To observe this we would like to remark that the inequalities in $P$ imply $|\zeta+\frac12|\leq \min(\frac12,\alpha_r+\frac12,\alpha_r+\alpha_s+\alpha_t+\frac12)$, so if, for example, in 1C the equation \eqref{eqvalipuniv} reduces to $|\zeta+\frac12|=\frac12$, as a corollary we obtain that $\frac12 = \min(\frac12,\alpha_r+\frac12,\alpha_r+\alpha_s+\alpha_t+\frac12)$.
From now on we can thus assume that $|\zeta+\frac12|=\min(\frac12,\alpha_r+\frac12,\alpha_r+\alpha_s+\alpha_t+\frac12)$ holds. 

It should now be observed that $P^{(0)}$ is the projection of $P$ to the space of $(\alpha_0,\ldots,\alpha_5)$. Thus now we have to show that for any vector in $P^{(0)}$, except in the interior of one of the $P_{II,t}$'s we do satisfy all conditions, while in the interior of the $P_{II,t}$'s one of the conditions fails. To show this we perform a case-by-case analysis. Our cases here depend on which of the following expressions (up to $S_4\times S_2$ symmetry) is minimal
\renewcommand{\theenumi}{\roman{enumi}}
\begin{enumerate}
\item $\frac12$;
\item $\alpha_0+\frac12$;
\item $\gamma_0+\frac12$;
\item $\alpha_0+\alpha_1+\alpha_2+\frac12$;
\item $\alpha_0+\alpha_1+\gamma_0+\frac12$;
\item $\alpha_0+\gamma_0+\gamma_1+\frac12$.
\end{enumerate} 
Note that $P_{II,t}$ is exactly the polytope in which $\alpha_t+\frac12$ is minimal. Thus we want to check that for
every vector in the facet of $P$ given by $term= |\zeta+\frac12|$ for term given by cases (i) and (iv)-(vi) above, the limits of the biorthogonal functions is $z$-dependent and the valuation of the bilinear form is correct (i.e. \eqref{eqvalipuniv} vanishes). Moreover we want to check that outside the interior of the facets of $P$ given by $\alpha_0+\frac12 =|\zeta+\frac12|$ or
$\gamma_0 + \frac12 = |\zeta+\frac12|$ the same holds, while in the interior either \eqref{eqvalipuniv} does not vanish (case (ii)), or the limits of the biorthogonal functions are not $z$-dependent (case (iii)).

In cases (i), (iv)-(vi) it turns out that we can identify which of the 20 (up to symmetry) linear pieces of \eqref{eqvalipuniv} we are (respectively 1C, 5C, 2B and 1A), and notice that the resulting linear equation is identically true. Moreover Proposition \ref{propzdep} shows that in these cases both biorthogonal functions have $z$-dependent limits, thus in these cases we are fine.

In case (iii), corresponding to the polytope $P_{II,4}$, we are in the linear piece 1A or 1B of \eqref{eqvalipuniv}, depending on the sign of $\gamma_0+\gamma_1$. In both cases we see that this linear function vanishes. However it is now obvious from Proposition \ref{propzdep} that in the interior of this polytope the first of the two biorthogonal functions has a $z$-independent limit (the second biorthogonal function has a $z$-dependent limit in this interior only if $\gamma_1 + |\zeta+\frac12| \geq \frac12$, or equivalently $\gamma_0+\gamma_1\geq 0$). Outside the interior the limits are all $z$-dependent.
%
%If $\frac12$ is the minimum and if this equals $|\zeta+\frac12|$, then we can easily determine we are in case 1C, and all vectors are good.
%If $\alpha_0+\alpha_1+\gamma_0+\frac12$ is minimal and equals $|\zeta+\frac12|$, we are in case 2B, and direct calculations shows all vectors are good. If $\alpha_0+\alpha_1+\alpha_2+\frac12$ is minimal, we are in case 5C. If $\alpha_0+\gamma_0+\gamma_1+\frac12$ is minimal we are in case 1A (though the condition on $\zeta$ here gives us a sub-polytope). If $\gamma_0+\frac12$ is minimal, we are in either case 1A, or 1B. But again a direct calculation shows that every instance of $\gamma_0+\frac12=|\zeta+\frac12|$ and $\gamma_0+\gamma_1\leq 0$ is covered in 1A, and every instance of $\gamma_0+\frac12=|\zeta+\frac12|$ and $\gamma_0+\gamma_1\geq 0$ is covered in 1B.

Finally case (ii), where $\alpha_0+\frac12$ is minimal and equals $|\zeta+\frac12|$ (i.e. we are in $P_{II,0}$). First note that the functions are clearly $z$-dependent outside the interior of $Q_0$. Let us now list all facets of $P_{II,0}$ up to the $(S_1\times S_3) \times S_2$ symmetry on the vectors $(\alpha_0,\ldots,\alpha_5)$. 
\begin{itemize}
\item $\alpha_0=-1/2$, we are in case 4A
\item $\alpha_0=0$, we are in case 1C
\item $\alpha_0=\alpha_1$, we are in case 2C
\item $\alpha_0=\gamma_0$, we are in case 1B
\item $1+\alpha_0=\alpha_3$, we are in case 3A
\item $1+\alpha_0=\gamma_1$, we are in case 4B/4D (in the intersection of these two)
\item $\alpha_1+\alpha_2=0$, we are in case 3C/5C  (in the intersection of these two)
\item $\alpha_1+\gamma_0=0$, we are in case 2B
\item $\gamma_0+\gamma_1=0$, we are in case 1A
\item $\alpha_2+\alpha_3=1$, we are in case 2A
\item $\alpha_3+\gamma_1=1$, we are in case 3B
\item $\gamma_0+\gamma_1=1$, we are in case 4C
\end{itemize}
In all these cases we see that the valuation equation \eqref{eqvalipuniv} linearizes and a direct calculation shows that it indeed vanishes. Thus outside the interior, all faces of $P_{II,0}$ correspond to a valid biorthogonal system. 

Finally we need to show that in the interior of $P_{II,0}$ the valuation equation \eqref{eqvalipuniv} does not vanish. Thus we check all 20 polytopes on which \eqref{eqvalipuniv} linearizes, add the equation $\alpha_0+\frac12 = |\zeta+\frac12|$ together with the vanishing of \eqref{eqvalipuniv}, and check whether these equations can be satisfied outside one of the facets of $P_{II,0}$. To save us some work we note that in $P_{II,0}$ we have $\alpha_1+\alpha_2\geq 0$, so we can only be in case 5 if we are on a boundary of that case which is also covered in case 3. Moreover we note that $\gamma_0 \leq \alpha_0+1=\frac12 + |\zeta+\frac12|$, so we can also only be on the boundary of case D which is also covered in case B. Thus we only need to consider the 12 cases 
1234ABC. Apart from case 1A the equation that \eqref{eqvalipuniv} vanishes is exactly the defining equation of a facet of $P_{II,0}$. In case 1A on the other hand, we see that \eqref{eqvalipuniv} vanishes as long as $\gamma_0+\gamma_1\leq 0$. As $\gamma_0+\gamma_1\geq 0$ is a bounding equation for $P_{II,0}$ this implies that we must have $\gamma_0+\gamma_1=0$, which again determines a facet of $P_{II,0}$. Thus we see that \eqref{eqvalipuniv} vanishes nowhere in the interior of $P_{II,0}$.
\end{proof}

We now know for which vectors the limit leads to two families of $z$-dependent functions with correct valuations so that we can hope that the limit from the elliptic hypergeometric level leads to a biorthogonality relation between the limiting functions. There are two more questions we need to answer. First of all we want to show that, indeed, for each of these vectors we do get a biorthogonal system. We do this by explicitly defining the limiting measures and showing that we really can take a limit in the biorthogonality relation at the elliptic hypergeometric level. Secondly we want to know what \textit{different} limits correspond to these vectors. As we saw in \cite{vdBR} taking limits along different vectors often give the same functions, so we want to tabulate what different limits we can obtain.

The polytope $P$ is just the intersection of the polytope $P^{(1)}$ with a given hyperplane ($\alpha_7+\alpha_8=0$), corresponding to the set of parameters in the $E_1$ of \eqref{eqrluniv}, thus the limit of the $E_1$ depends only on the face of $P$ the vector $(\alpha;\zeta)$ is contained in, and different faces of $P$ lead to different limits of the $E_1$. The polytope $P^{(0)}$ is the projection of the polytope $P$ by ignoring $\zeta$. In this projection the inverse image of a face of $P^{(0)}$ correspond to a union of faces in $P$, and not always a single face of $P$. In particular, limits associated to vectors in the same face of $P^{(0)}$ can still have different values. The point is that in the equation  $|\zeta+\frac12|=\min(\frac12,\frac12 + \alpha_r, \frac12+\alpha_r+\alpha_s+\alpha_t)$, $\zeta$ does not depend linearly on the vector $\alpha$ within each face of $P^{(0)}$. However, if we restrict ourselves to the different facets of $P$ given by $\zeta+\frac12=\frac12$, $\zeta+\frac12=\frac12+\alpha_r$ (for $0\leq r\leq 5$) and $\zeta+\frac12=\frac12+\alpha_r+\alpha_s+\alpha_t$ (for $0\leq r<s<t\leq 5$), and take the projections of these facets (called resp. $P_{I}$, $P_{II,t}$ (the same as above) and $P_{III,(r,s,t)}$), the different faces in the projection will lead to different limits for the $E_1$. By inspection we can see that the limits of the prefactor in \eqref{eqrluniv} also only depend on which face of one of these projections the vector is in. Therefore we see that 
the limit of the biorthogonal functions and their squared norm formula only depend on which face of $P_{I}$, $P_{II,t}$ and $P_{III,(r,s,t)}$ contains the vector $\alpha$ (if we assume $\zeta$ is given by the unique value for which we can have a biorthogonal system as limit). It could conceivably be that while the limits of the $E_1$'s and of the prefactors is different in two different faces, the product of the limit of the $E_1$ and the prefactor (i.e. the limit of the biorthogonal functions themselves) is the same on two different faces. By inspecting the different limits associated to the faces and considering basic properties of these limits (symmetries, locations of poles), we can exclude this possibility. Thus we obtain the following theorem.
\begin{theorem}\label{thmwhensystem}
Define the polytopes $P_I$, $P_{II,t}$ ($0\leq t\leq 5$) and $P_{III,(r,s,t)}$ ($0\leq r<s<t\leq 5$) as in \cite{vdBR}, that is 
\begin{itemize}
\item $P_{I}$ is given by the bounding equations
\[
\alpha_r \geq 0, \qquad \sum_{r=0}^5 \alpha_r = 1.
\]
\item $P_{II,t}$ is given by the bounding equations
\[
-\frac12 \leq \alpha_t\leq 0, \qquad 
\alpha_t \leq \alpha_r \leq 1+\alpha_t \quad (r\neq t), \qquad 
0\leq \alpha_r + \alpha_s \leq 1 \quad (r,s\neq t), \qquad 
\sum_{r=0}^5 \alpha_r = 1.
\]
\item $P_{III,(r,s,t)}$ is given by the bounding equations 
\[
\alpha_a+\alpha_b \leq 0 \ \ (a, b \in \{r,s,t\}), \quad
-(\alpha_r+\alpha_s+\alpha_t) \leq \alpha_a \leq 1+(\alpha_r+\alpha_s+\alpha_t) \ \  (a\not \in \{r,s,t\}), \quad
\sum_{a=0}^5 \alpha_a =1.
\]
\end{itemize}

Then these polytopes tile $P^{(0)}$. The leading coefficients of the triple
\[
(\tilde R_l(z;t_r;u_0,u_1;z), \tilde R_l(z;t_r;u_1,u_0), \langle  R_l(\cdot;t_r;u_0,u_1), R_l(\cdot;t_r;u_1,u_0)\rangle_{t_r,u_r}),
\]
when we replace $t_r\to t_rp^{\alpha_r}$, $u_r \to p^{\alpha_{r+4}}$ and $z\to p^{\zeta}$ when $\zeta$ is determined as a function of $\alpha$ by $\zeta+\frac12 = \min(\frac12, \frac12 +\alpha_r, \frac12 + \alpha_r+\alpha_s+\alpha_t)$, for a vector $\alpha \in P^{(0)}$ is determined by the face of the tiling of $P^{(0)}$ given above which contains $\alpha$. 

If two faces are related by a shift from the translation group $T$ from Proposition \ref{proppolytopes} the corresponding triples have identical leading coefficients. If two faces are related by the action of $S_4\times S_2$ on $P^{(0)}$ the leading coefficients are equal up to the same permutation of the parameters $t_r$ and $u_r$. If the faces are related by the flip of Proposition \ref{proppolytopes} (followed by a translation in $T$), the leading coefficients of the biorthogonal functions are related by setting $t_r\to 1/t_r$, $u_r\to 1/u_r$ and $q\to 1/q$, in this case the squared norms are identical.
\end{theorem}
We are left with the problem that we want to find (at least one) explicit bilinear form turning each of these cases into a biorthogonal system. It turns out that in \cite{vdBR} we determined limits of the elliptic beta integral evaluation (i.e. the equation $\langle 1,1\rangle =1$) for parameters specialized like this as $p\to 0$. The kind of limit depended precisely on the face the vector $\alpha$ was in in the tiling of $P^{(0)}$ as in the above theorem. These limits all essentially involved evaluating the integrand at $zp^{\zeta}$ for $z$ in some set which is independent of $p$ and $\zeta$ such that $|\zeta+\frac12| = \min(\frac12,\alpha_r+\frac12,\alpha_r+\alpha_s+\alpha_t + \frac12)$, that is, exactly where we want to evaluate our biorthogonal functions. 
It turns out that the exact same limits are still valid if we plug in the biorthogonal functions. Thus in all the cases mentioned in the theorem above we obtain at least an algebraic measure giving biorthogonality. In the next section we will describe these limiting measures briefly. 

In the cases of biorthogonal functions associated to a vector not contained in any translate of $P^{(0)}$, that is, the flips of $0022pp$, $04as$ (i.e. Continuous $q^{-1}$-Hermite), and $0031as$ (i.e. Stieltjes-Wiegert),  we can obtain measures by taking limits outside this polytope. These measures are related to bilateral series. See \cite[Section 5]{vdBRmeas} for details.

It should be noted that the measures obtained are not necessarily positive. Of course you would only hope this were the case if you were considering orthogonal polynomials (i.e. cases where the limits of both families of biorthogonal functions are identical) and the squared norms were positive. However even in those cases the measures we find are not always positive. We believe you can find positive measures in those cases by taking limits from the elliptic level by looking at limits outside the polytope as in \cite[Section 9]{vdBR}. 

Finally we should make a few remarks about explicitly obtaining the limits of the biorthogonal functions. Our description so far shows that we can obtain a limit by first rewriting the biorthogonal functions as beta integrals. The resulting limit will often again be a singular integral, which reduces to a finite sum of residues. This is far from our desired method of looking at the defining sum \eqref{equnivbiortho} for the biorthogonal functions and replacing sum and limit. Unfortunately the latter method does not always work (essentially because the valuation of the sum is more than the valuation of the summands). However due to the $W(E_7)$-symmetry of the elliptic hypergeometric series there are many different ways of writing the biorthogonal functions as a series. In \cite{vdBRmult} we show that 8 of these representations generalize to the multivariate level, and at least one of those admits interchanging sum and limit. In particular, if we want to do explicit calculations we do not have to make the detour through singular beta integrals in order to get expressions for the limiting biorthogonal functions.

Finding all different biorthogonal systems of rational hypergeometric functions which are limits of the elliptic hypergeometric biorthogonal functions has now been reduced to a combinatorial exercise of writing down the different faces of the given tiling of $P^{(0)}$ modulo the symmetry group $S_4\times S_2$. Each such face corresponds to a pair of limiting biorthogonal functions and a corresponding limit of the measure. If two faces are related by a translation from $T$ in Proposition \ref{proppolytopes} the resulting biorthogonal functions are identical, but we will obtain different measures. If they are related using the flip of Proposition \ref{proppolytopes} we again have the same functions, but different measures, some for $|q|<1$ and others for $|q|>1$. In Appendix \ref{secsystemlimit} we give a complete list of all different biorthogonal systems we obtain in this way and describe how they are related using these two kinds of symmetries. 

\section{Measures}
In this section we briefly list the measures associated to vectors within the polytope $P$. The results in here are slight generalizations of the results in \cite{vdBR}, where we considered the limits of, essentially, $\langle 1,1\rangle$. Plugging in functions does not significantly alter the arguments which shows that these limits are valid, as the functions we consider converge uniformly on compact sets outside some isolated points, by definition. \cite{vdBRmeas} is devoted to extending this section to the multivariate case, and includes detailed proofs of the results.

One important idea is to remember that there exist $m_f$ and $m_g$ such that 
\[
\hat f(z) := \frac{\Gamma(u_0 z^{\pm 1})}{\Gamma(u_0 q^{-m_f} z^{\pm 1})} f(z), \qquad 
\hat g(z):=\frac{\Gamma(u_1 z^{\pm 1})}{\Gamma(u_1 q^{-m_g} z^{\pm 1})} g(z)
\]
are holomorphic. The bilinear form can then be expressed as 
\begin{align*}
\langle  f,g\rangle_{t_0,t_1,t_2,t_3,t_4,t_5:q;p}  &= 
\frac{(q;q) (p;p) }{2 \prod_{0\leq r<s\leq 5} \Gamma(t_rt_s;p,q)} 
\\ & \qquad \times
 \int_{C}
\hat f(z) \hat g(z)
 \frac{\prod_{r=0}^3 \Gamma(t_rz^{\pm 1};p,q)
 \Gamma(u_0 q^{-m_f} z^{\pm 1},u_1 q^{-m_g} z^{\pm 1}) }{\Gamma(z^{\pm 2};p,q)} \frac{dz}{2\pi i z},
\end{align*}
that is, the same measure with slightly different parameters, multiplied by a holomorphic uniformly converging (as $p\to 0$) function. For the integral limits below one will only see this idea reflected in the choice of contour, but for the series limits we need to actually consider these new functions. It is useful to introduce the notation 
\[
\tilde t_r := t_r, \quad (0\leq r\leq 3), \qquad 
\tilde t_4 := \tilde u_0 := q^{-m_f} u_0, \qquad
\tilde t_5 := \tilde u_1 :=  q^{-m_g} u_1,
\]
for these new parameters. Moreover we introduce $m_i$ for $0\leq i\leq 5$ as $m_i=0$ if $0\leq i\leq 3$ and $m_4=m_f$ and $m_5=m_g$.
 Notice that the $lc(f)$ and $lc(\hat f)$ are related via 
\[
lc(\hat f)(z) = lc(f)(z)
%\left( 
(\tilde u_0/z;q)_{m_f}%\right)
^{1_{\gamma_0-\zeta\in \mathbb{Z}}} 
%\left(
(\tilde u_0 z;q)_{m_f}% \right)
^{1_{\gamma_0+\zeta\in \mathbb{Z}}}
\left( (-\frac{qz}{u_0})^{m_f} q^{\binom{m_f}{2}} \right)^{\lfloor \gamma_0-\zeta\rfloor}
\left( (-\frac{q}{u_0z})^{m_f} q^{\binom{m_f}{2}} \right)^{\lfloor \gamma_0+\zeta\rfloor}.
\]
Since $\hat f$ is holomorphic, we know that $lc(\hat f)$ is holomorphic, so this equation gives the possible locations of poles of $lc(f)$.

The first proposition is the analogue of \cite[Proposition 4.1]{vdBR}, and describes Nasrallah-Rahman type bilinear forms, 
\begin{proposition}\label{propPI}
Choose generic parameters satisfying $\prod_r t_r=q$.
Let $\alpha \in \mathbb{R}^6$, $\sum_{r=0}^5 \alpha_r=1$ and $\alpha_r\geq 0$ for $0\leq r\leq 5$. Let $f\in \tilde A(t_4p^{\alpha_4})$ and $g\in \tilde A(t_5p^{\alpha_5})$.

%Let $m_f$ and $m_g$ be such that $\theta(pq z^{\pm 1}/t_4;p,q)_{m_f} f(z)$ and $\theta(pq z^{\pm 1}/t_5;p,q)_{m_g} g(z)$ are holomorphic, and set $\tilde t_r=t_r$ for $0\leq r \leq 3$, $\tilde t_4=q^{-m_f} t_4$, and $\tilde t_5=q^{-m_g} t_5$.

We now have the limit
\begin{align*}
\lim_{p\to 0} p^{-val(f)-val(g)}  \langle f(z;t_rp^{\alpha_r}),g(z;t_rp^{\alpha_r})\rangle_{t_r p^{\alpha_r}} & = \frac{(q;q) \prod_{0\leq r<s\leq 5: \alpha_r+\alpha_s=0} (t_rt_s;q)}
{2 \prod_{0\leq r<s\leq 5:\alpha_r+\alpha_s=1}(q t_r^{-1}t_s^{-1});q)} \\ &\qquad  \times 
\int_{C} lc(f)(z) lc(g)(z) 
 \frac{(z^{\pm 2};q) \prod_{r:\alpha_r=1} (q t_r^{-1} z^{\pm 1};q)}{\prod_{r:\alpha_r=0} (t_rz^{\pm 1};q)} \frac{dz}{2\pi i z},
\end{align*}
where $lc(f)=lc(f(z;p^{\alpha_r}t_r))$ and likewise for $lc(g)$.
Here the integration contour $C=C^{-1}$ is such that it includes the points $q^j \tilde t_r$, (for $0\leq r\leq 5$ with $\alpha_r=0$ and $j\geq 0$) and  excludes their reciprocals. The contour can be taken to be the unit circle if $|\tilde t_r|<1$ for all $r$ with $\alpha_r=0$.
\end{proposition}
Notice that, while it is obvious that we have to consider the leading coefficient of $f$ and $g$ as the $t_r$ change with $p$ according to $p^{\alpha_r} t_r$, we leave the $z$'s invariant with $p$, i.e., we have $\zeta=0$. Note that we therefore exactly have that $|\zeta+\frac12| = \min(\frac12,\frac12 + \alpha_r, \frac12+\alpha_r+\alpha_s+\alpha_t)$, i.e. we evaluate the functions at the right $\zeta$, given $\alpha$ (according to Theorem \ref{thmwhensystem}). In the following propositions we will have different values of $\zeta$, but one can easily check that it is always the right value. 

By breaking the $z\to 1/z$ symmetry in the integral before taking limits we can obtain limits for different values of $\alpha$. This gives first of all (compare with \cite[Proposition 4.2]{vdBR}).
\begin{proposition}\label{proplimip3}
Let $t_r\in \mathbb{C}$ be generic, satisfying $\prod_r t_r=q$. Let $0\leq a,b,c\leq 5$ and let $\alpha \in \mathbb{R}^6$, and $-\frac12\leq \zeta<0$ satisfy $\sum_{r=0}^5 \alpha_r=1$, $\alpha_a+\alpha_b+\alpha_c =\zeta$ and
$\zeta\leq \alpha_r\leq -\zeta$ for $r=a,b,c$ and
$-\zeta\leq \alpha_r\leq 1+\zeta$ for $r\neq a,b,c$. Let
$f\in \tilde A(t_4p^{\alpha_4})$ and $g\in \tilde A(t_5p^{\alpha_5})$. 
Then we have

\begin{align*}
\lim_{p\to 0} & p^{-val(f)-val(g)}\langle f(z;t_r p^{\alpha_r}) ,g(z;t_rp^{\alpha_r}) \rangle_{t_r p^{\alpha_r}} \\ & = 
\frac{(q;q) \prod_{\substack{r,s\in \{a,b,c\}\\ \alpha_r+\alpha_s=-1}} (t_rt_s;q) 
\prod_{\substack{r \in \{a,b,c\}, s\not \in \{a,b,c\} \\ \alpha_r+\alpha_s=0}} (t_rt_s;q)
 }{ \prod_{\substack{r,s\in \{a,b,c\}\\ \alpha_r+\alpha_s=0}} (q t_r^{-1} t_s^{-1};q) 
\prod_{\substack{r,s\\ \alpha_r+\alpha_s=1}} (q t_r^{-1}t_s^{-1};q)}
\\& \qquad \times 
\int_{C} lc(f)(z) lc(g)(z) 
\frac{\prod_{\substack{r\in \{a,b,c\} \\ \alpha_r=- \zeta}} (q/t_rz;q) \prod_{\substack{r\not\in \{a,b,c\} \\ \alpha_r=1+\zeta}} (qz/t_r;q)}
{\prod_{\substack{r\in \{a,b,c\} \\  \alpha_r=\zeta}} (t_r/z;q) \prod_{\substack{r\not \in \{a,b,c\} \\ \alpha_r=-\zeta}} (t_rz;q)}
\left( \frac{ (z^2;q) \prod_{\substack{r\in \{a,b,c\} \\ \alpha_r=1/2}} (qz/t_r;q)}{(qz^2;q)\prod_{\substack{r\in \{a,b,c\} \\ \alpha_r=-1/2}} (t_rz;q)}  \right)^{1_{\zeta=-1/2}}
\\ & \qquad \qquad \times \theta(qz/t_at_bt_c;q)
\frac{dz}{2\pi i z},
\end{align*}
where $lc(f)=lc(f(p^{\zeta}z))$ and similarly for $g$. We also have the usual conditions on the integration contour (i.e. if $(t_rz;q)$ is in the denominator of the measure, we include poles of the form $q^k/\tilde t_r$, $k\geq 0$, and if 
$(t_r/z;q)$ is in the denominator we exclude poles of the form $\tilde t_r q^{-k}$, $k\geq 0$).
\end{proposition}

The next case we consider is an integral as in \cite[Proposition 4.4]{vdBR}. 
%The second expression as a double sum can be obtained from the integral expression by taking residues while shrinking the contour to the origin. In the case $a \in \{4,5\}$ and/or $b\in \{4,5\}$ we must take residues of $lc(f)$ and/or $lc(g)$. To nonetheless give a uniform formula we rewrite $lc(f)$ in terms of the holomorphic function $lc(\hat f)$, so that residues of $lc(f)$ correspond to evaluations of $lc(\hat f)$. 
\begin{proposition}\label{proplimip2}
Let $t_r\in \mathbb{C}$ be generic such that $\prod_r t_r=q$. 
Let $0\leq a,b\leq 5$ and let $\alpha \in \mathbb{R}^6$, and $-\frac12\leq \zeta<0$ satisfy $\sum_{r=0}^5 \alpha_r=1$, $\alpha_a=\alpha_b =\zeta$ and $-\zeta\leq \alpha_r\leq 1+\zeta$ for $r\neq a,b$. Let $f\in \tilde A(t_4p^{\alpha_4})$ and $g\in \tilde A(t_5p^{\alpha_5})$

Then we have the limit (where $C$ satisfies the usual conditions on a contour)
\begin{align*}
&\lim_{p\to 0} p^{-val(f)-val(g)}  \langle f(z;t_r p^{\alpha_r})  ,g(z;t_rp^{\alpha_r}) \rangle_{t_r p^{\alpha_r}}   
\\ &= 
\frac{(q;q) (t_at_b;q)^{1_{\zeta=-1/2}} 
\prod_{r:\alpha_r=-\zeta} ( t_rt_a,t_rt_b;q) }{   \prod_{\substack{0\leq r<s\leq 5 \\ \alpha_r+\alpha_s=1}} (qt_r^{-1}t_s^{-1};q) } \\ &\qquad  \times 
\int_{C} lc(f)(z) lc(g)(z) \frac{\prod_{r:\alpha_r=1+\zeta} (qz/t_r;q)}{(t_a/z,t_b/z;q) \prod_{r: \alpha_r=-\zeta} (t_rz;q)}
\left( \frac{1-z^2}{(t_az,t_bz;q)}  \right)^{1_{\zeta=-1/2}} \frac{ \theta(wz, \frac{qz}{t_at_bw};q) }{\theta(t_aw, t_bw;q)}
\frac{dz}{2\pi i z}
\end{align*}
where $w$ is some new parameter (which does not affect the value of the bilinear form when applied to our functions). Here $lc(f)=lc(f(zp^{\zeta}))$ and likewise for $lc(g)$.% and likewise for $lc(g)$, $lc(\hat f)$ and $lc(\hat g)$. 
\end{proposition}
There also exists a double series expression for this bilinear form. The double series expression is
the extensions to a measure of some of the cases in  \cite[Proposition 4.3]{vdBR}. It can be obtained from the integral expression by taking residues while shrinking the contour to the origin. In the case $\{a,b\} \cap \{4,5\} \neq \emptyset$ we get extra complications, as we would have to consider residues of $lc(f)$ and $lc(g)$. Thus  our proposition statement will exclude those cases, though after the proposition we will explain how to get a more general expression.
\begin{proposition}\label{proplimip2b}
Let $t_r\in \mathbb{C}$ be generic such that $\prod_r t_r=q$. 
Let $0\leq a,b\leq 3$ and let $\alpha \in \mathbb{R}^6$, and $-\frac12\leq \zeta<0$ satisfy $\sum_{r=0}^5 \alpha_r=1$, $\alpha_a=\alpha_b =\zeta$ and $-\zeta\leq \alpha_r\leq 1+\zeta$ for $r\neq a,b$ ($0\leq r\leq 5$). Let $f\in \tilde A(t_4p^{\alpha_4})$ and $g\in \tilde A(t_5p^{\alpha_5})$

Then we have the limit 
\begin{align*}
&\lim_{p\to 0} p^{-val(f)-val(g)}  \langle f(z;t_r p^{\alpha_r})  ,g(z;t_rp^{\alpha_r}) \rangle_{t_r p^{\alpha_r}}   
\\ &= \left( \frac{1}{(qt_a^2;q)} \right)^{1_{\zeta=-\frac12}} 
 \frac{\prod_{r: \alpha_r=-\zeta} (t_rt_b;q)\prod_{r:\alpha_r=1+\zeta} (\frac{q t_a}{t_r};q) }
{ (\frac{ t_b}{ t_a};q)\prod_{r,s:\alpha_r+\alpha_s=1}(\frac{q}{t_rt_s};q)}
\\ & \qquad  \times \sum_{k \geq 0}
lc(f)(t_a q^k) lc(g)(t_a q^k) 
\left( \frac{(q t_a^2;q)_{2k}(t_a^2 , t_a t_b;q)_{k}}{( t_a^2;q)_{2k}} \right)^{1_{\zeta=-\frac12}}
\frac{\prod_{\alpha_r=-\zeta}(t_r  t_a ;q)_{k}}{(q,\frac{q  t_a}{t_b};q)_{k}  \prod_{r:\alpha_r=1+\zeta} (\frac{q t_a}{t_r};q)_{k} }  q^{k}
\\ & + 
\left( \frac{}{(qt_b^2;q)} \right)^{1_{\zeta=-\frac12}} 
 \frac{\prod_{r: \alpha_r=-\zeta} (t_rt_a;q)\prod_{r:\alpha_r=1+\zeta} (\frac{q t_b}{t_r};q)}
{ (\frac{t_a}{ t_b};q)\prod_{r,s:\alpha_r+\alpha_s=1}(\frac{q}{t_rt_s};q)}
\\ & \qquad  \times \sum_{k \geq 0}
lc(f)(t_b q^k) lc(g)(t_b q^k) 
\left( \frac{(q t_b^2;q)_{2k}(t_b^2, t_b t_a;q)_{k}}{( t_b^2;q)_{2k}} \right)^{1_{\zeta=-\frac12}}
\frac{\prod_{\alpha_r=-\zeta}(t_r  t_b ;q)_{k}}{(q,\frac{q  t_b}{t_a};q)_{k}  \prod_{r:\alpha_r=1+\zeta} (\frac{q t_b}{t_r};q)_{k} }  q^{k}
\end{align*}
where $lc(f)=lc(f(zp^{\zeta}))$ and likewise for $lc(g)$.
\end{proposition}
To extend the result to the case $a=4$ we should consider the series based around $t_aq^k$ as a bilateral series which happens to have a term $(q;q)_k$ in the denominator which makes it zero for $k<0$. However, if $a=4$ we would be evaluating $lc(f)$ at one of its poles (which are located at $t_aq^k$ for $-m_f\leq k<0$). As long as $lc(f)$ indeed has a pole, it would cancel the zero from the $(q;q)_k$-term, and we would get an extra term. In particular we get at most $m_f$ extra terms. To make this more formal one could rewrite the summand in terms of $lc(\hat f)$ instead of $lc(f)$, so that we can evaluate $lc(\hat f)$ in the relevant points. It should be noted that in this case the support of the measure becomes $u_0 q^{\mathbb{Z}}$, not just $u_0q^{\mathbb{Z}_{\geq 0}}$, plus the support coming from the second series.

The final limit case for the continuous measure corresponds to single sum expressions. Many of the issues surrounding the double sum expression arise arise here as well. Compare this proposition to the single series cases of \cite[Proposition 4.3]{vdBR}.
\begin{proposition}\label{propsumlim}
Let $t_r\in \mathbb{C}$ be generic such that $\prod_r t_r=q$.
Let $\sum_{r=0}^5 \alpha_r=1$. Let $0\leq a\leq 3$ be such that $-\frac12\leq \alpha_a<0$ and $1+\alpha_a\geq \alpha_r > \alpha_a$ for $r\neq a$, $0\leq r\leq 5$ and such that 
$1\geq \alpha_r+\alpha_s> 0$ for $r,s\neq a$. 
Moreover assume $2\alpha_a = \sum_{r\neq a: \alpha_r+\alpha_a <0} (\alpha_r+\alpha_a)$.
Let $f\in \tilde  A(t_4p^{\alpha_4})$
and likewise $g\in \tilde A(t_5p^{\alpha_5})$. 
%
%Write $\tilde t_r= t_r$ if $r\neq a$ or $r=0,1,2,3$ and $\tilde t_4 = q^{-m_f} t_4$ if $a=4$ and $\tilde t_5 = q^{-m_g} t_5$ if $a=5$. Moreover set $\hat f=f$ unless $a=4$, in which case
%\[
%\hat f (z) = \theta(pq z^{\pm 1}/t_4;p,q)_{m_f} f(z)
%\]
%and likewise $\hat g=g$ unless $a=5$, in which case
%\[
%\hat g (z) = \theta(pq z^{\pm 1}/t_5;p,q)_{m_g} g(z).
%\]

We have
%\begin{align*}
%\lim_{p\to 0} & p^{-val(f)-val(g)} \langle f(z;t_rp^{\alpha_r}),g(z;t_rp^{\alpha_r}) \rangle_{t_r p^{\alpha_r}} \\ & 
%=
%\frac{\prod_{r,s: \alpha_r+\alpha_s=0} (t_rt_s;q)
%\prod_{r: \alpha_r=\alpha_a+1} (q\tilde t_a/\tilde t_r;q)}{\prod_{r,s:\alpha_r+\alpha_s=1} (q/t_rt_s;q)
%\prod_{r:\alpha_r=-\alpha_a} (\tilde t_r \tilde t_a;q)}
%%\prod_{r:\alpha_r=-\alpha_a} \frac{(t_rt_a;q)}{(\tilde t_r \tilde t_a;q)}
%\prod_{r\neq a:\alpha_r+\alpha_a<0} \left( - \frac{1}{\tilde t_r \tilde t_a} \right)^{m_r+m_a} q^{-\binom{m_r+m_a}{2}}
%\left( \frac{1}{(q\tilde t_a^2;q)}\right)^{1_{\alpha_a=-\frac12}} 
%\\ & \qquad \times \sum_{k\geq 0}  lc(\hat f)(\tilde t_a q^k) lc(\hat g)(\tilde t_a q^k)
%\left( \frac{1-\tilde t_a^2 q^{2k}}{1-\tilde t_a^2}  (\tilde t_a^2;q)_k   \right)^{1_{\alpha_a=-\frac12}}
%\\ & \qquad \qquad \times 
% \frac{\prod_{r:\alpha_r=-\alpha_a} (\tilde t_r \tilde t_a;q)_k}{(q;q)_k\prod_{r:\alpha_r=\alpha_a+1}(q\tilde t_a/\tilde t_r;q)_k}
%\left( (-1)^k q^{ \binom{k}{2}} \right)^{N-2}
%\left( \tilde t_a^{N-2}  \prod_{r\neq a: \alpha_r+\alpha_a<0} \tilde t_r  \right)^{k} 
%\end{align*}
\begin{align*}
\lim_{p\to 0} & p^{-val(f)-val(g)} \langle f(z;t_rp^{\alpha_r}),g(z;t_rp^{\alpha_r}) \rangle_{t_r p^{\alpha_r}} \\ & 
=
\frac{\prod_{r,s: \alpha_r+\alpha_s=0} (t_rt_s;q)
\prod_{r: \alpha_r=\alpha_a+1} (q t_a/ t_r;q)}{\prod_{r,s:\alpha_r+\alpha_s=1} (q/t_rt_s;q)
\prod_{r:\alpha_r=-\alpha_a} ( t_r  t_a;q)}
\left( \frac{1}{(qt_a^2;q)} \right)^{1_{\alpha_a=-\frac12}} 
\\ & \qquad \times \sum_{k\geq 0}  lc(f)(t_a q^{k}) lc( g)(t_a q^{k})
\left( \frac{1- t_a^2 q^{2k}}{1- t_a^2}  (t_a^2;q)_{k}   \right)^{1_{\alpha_a=-\frac12}}
\\ & \qquad \qquad \times 
 \frac{\prod_{r:\alpha_r=-\alpha_a} ( t_r  t_a;q)_{k}}{(q;q)_{k}\prod_{r:\alpha_r=\alpha_a+1}(q t_a/t_r;q)_{k}}
\left( (-1)^{k} q^{ \binom{k}{2}} \right)^{N-2}
\left(  t_a^{N-2}  \prod_{r\neq a: \alpha_r+\alpha_a<0}  t_r  \right)^{k} 
\end{align*}
where $N=|\{ r~|~ r\neq a,  \alpha_r<-\alpha_a\}|$, 
and $lc(\hat f)=lc(\hat f(zp^{\alpha_a}))$, and likewise for $lc(\hat g)$ and the valuations.
\end{proposition}
%There exists a similar expression if the $\alpha_r$ are such that we have $\alpha_r+\alpha_s=0$ for some $r,s\neq a$. However in this case the limit is only valid under an extra condition on the functions $f$ and $g$ (to ensure convergence). In particular if we want to plug in our elliptic hypergeometric biorthogonal 
%functions $(\langle \tilde R_{n},\tilde R_{m} \rangle$), this limit would for some values of the parameters only hold for small $n$ and $m$. Fortunately those cases are also treated in Proposition \ref{proplimip3}.
%
%It should also be noted that if we have one parameter $\alpha_r=\alpha_a$ we can get a discrete measure as limit, which is a double sum. This proposition should also be compared to \cite[Proposition 4.3]{vdBR}.

Finally let us consider finite measures. In this limit we want to insist that $t_0t_1=q^{-N}$ for all $p$, so $\alpha_0+\alpha_1=0$. The proof in this case is just interchanging sum and limit. Note that the resulting expression looks very much like the sum cases above.
\begin{proposition}
Suppose $t_r\in \mathbb{C}$ are generic, such that $t_0t_1=q^{-N}$ and $t_2t_3t_4t_5=q^{N+1}$. Suppose $-\frac12\leq \alpha_0=-\alpha_1\leq 0$, and $1+\alpha_0\geq \alpha_r\geq \alpha_0$ for $r>1$, and that $\alpha_r+\alpha_s\leq 1$ (for $2\leq r<s\leq 5$) and 
$\alpha_2+\alpha_3+\alpha_4+\alpha_5=1$ and $\sum_{2\leq r\leq 5:\alpha_r<-\alpha_0} \alpha_0+\alpha_r=2\alpha_0$. 
Let $f\in \tilde  A(t_4p^{\alpha_4})$
and likewise $g\in \tilde A(t_5p^{\alpha_5})$.
Then $\langle f,g \rangle_{t,u} \in M(\tilde t, \tilde u,q,t)$ and
\[
\lim_{p\to 0} p^{-val(f(zp^{\alpha_0}))-val(g(zp^{\alpha_0}))}\langle f,g \rangle_{t} 
=
\sum_{k=0}^N
lc(f)(t_0 q^{k}) lc(g)(t_0 q^{k}) w_{k,\alpha}(t_r),
\]
with $lc(f)=lc(f(zp^{\alpha_0}))$ and $lc(g)=lc(g(zp^{\alpha_0}))$.
%, and the weights $w_{k,\alpha}$ satisfy
%\[
%\sum_{k=0}^N w_{k,\alpha}=1.
%\]
Here the weights $w_{k,\alpha}$ are given by
\begin{itemize}
\item If $\alpha_0=0$ 
\begin{align*}
w_{k,\alpha} % (t_0, t_1,t_2,t_3,t_4,t_5) 
 & = 
 \frac{1-t_0^2 q^{2k}}{1-t_0^2} \frac{(q^{-N},t_0^2;q)_k}{(q,qt_0/t_1;q)_k} \left(\frac{1}{t_1t_0^3 q}\right)^k q^{-2\binom{k}{2}}
 \frac{ 1 }{(t_1/t_0;q)_N }
%\\&\qquad \times 
\prod_{2\leq r\leq 5:\alpha_r=0}
\frac{(t_0t_r;q)_k (t_1t_r;q)_N}{(qt_0/t_r;q)_k}
(-qt_0/t_r)^{k} q^{\binom{k}{2}} 
\\ & \qquad\times 
\prod_{2\leq r\leq 5:\alpha_r=1}
\frac{(t_0t_r;q)_k (t_1t_r;q)_N}{(qt_0/t_r;q)_k}
(-t_0t_r)^{-k} (-t_1t_r)^{-N} q^{-\binom{k}{2}-\binom{N}{2}}
%\\ & \qquad\times
\prod_{2\leq r<s\leq 5: \alpha_r+\alpha_s=1} \frac{1}{(q/t_rt_s;q)_N} 
\end{align*}

\item If $-1/2<\alpha_0<0$ 
\begin{align*}
w_{k,\alpha} %(t_0, t_1,t_2,t_3,t_4,t_5) 
&=  \frac{(q^{-N};q)_k  }{(q;q)_k  t_0^{2k} q^{2\binom{k}{2}}}
%\\ & \qquad \times
 \prod_{2\leq r\leq 5:\alpha_r=\alpha_0} \frac{ (qt_0^2)^{k} q^{2\binom{k}{2}} }{ (qt_0/t_r;q)_k}(t_1t_r;q)_N
%\\ & \qquad \times 
\prod_{2\leq r\leq 5:\alpha_0<\alpha_r<-\alpha_0}   (-t_0t_r)^{k} q^{\binom{k}{2}} 
\\ & \qquad \times 
\prod_{2\leq r\leq 5:\alpha_r=-\alpha_0} (t_0t_r;q)_k
%\\ & \qquad \times 
\prod_{2\leq r\leq 5:\alpha_r=1+\alpha_0} \frac{(qt_0/t_r;q)_N}{(qt_0/t_r;q)_k}
%\\ & \qquad \times
\prod_{2\leq r<s\leq 5: \alpha_r+\alpha_s=1} \frac{1}{(q/t_rt_s;q)_N} 
\end{align*}

\item If $\alpha_0=-1/2$ 
\begin{align*}
w_{k,\alpha} %(t_0, t_1,t_2,t_3,t_4,t_5)
& = 
\frac{1-t_0^2 q^{2k}}{1-t_0^2} \frac{(q^{-N},t_0^2;q)_k}{(q,qt_0/t_1;q)_k} 
\frac{1}{(t_1/t_0;q)_N}
\left(\frac{qt_0}{t_1}\right)^k  \left(-\frac{t_1}{t_0}\right)^{N} q^{2\binom{k}{2}+\binom{N}{2}} 
\\ & \qquad \times 
\prod_{2\leq r\leq 5: \alpha_r=-1/2} \frac{(t_0t_r;q)_k (t_1t_r;q)_N (-qt_0/t_r)^{k} q^{\binom{k}{2}} } {(qt_0/t_r;q)_k  }
\\ & \qquad \times 
\prod_{2\leq r\leq 5: \alpha_r=1/2} \frac{(t_0t_r;q)_k (t_1t_r;q)_N}{(qt_0/t_r;q)_k  } (-t_0t_r)^{-k}(-t_1t_r)^{-N} q^{-\binom{k}{2}-\binom{N}{2} } 
%\\ & \qquad \times 
\prod_{\substack{2\leq r<s \leq 5 \\ \alpha_r+\alpha_s=1}} \frac{1}{(q/t_rt_s;q)_N} 
\end{align*}
\end{itemize}
\end{proposition}

\section{The $q$-Askey scheme}\label{secAWscheme}
In this section we want to highlight the part of our degeneration scheme of biorthogonal families which corresponds to the $q$-Askey scheme. The $q$-Askey scheme is a unified way of presenting the classical $q$-hypergeometric orthogonal families of polynomials, and put them in a picture which clarifies the possible limit transitions. Orthogonal polynomials are just the special case of biorthogonal rational functions, where the poles of the rational functions have disappeared to infinity and the a priori two different families on both sides of the bilinear form have reduced to identical functions. Therefore it comes as no surprise that the different families in the $q$-Askey scheme return as special cases in our degeneration scheme of biorthogonal families of rational $q$-hypergeometric functions. 

We would like to emphasize however that we have a slightly different philosophy on what we include in our scheme and what we consider to be different. In particular we do not care about positivity of either the squared norms, or the measure, whereas the $q$-Askey scheme does. It only makes sense to worry about this in the context of orthogonality (as opposed to biorthogonality), but in the case of orthogonality it is a very important property. Because of this difference we include families which have no positive measure at all, while the $q$-Askey scheme would consider two families of orthogonal polynomials to be different if they are the same functions, but in different ranges of the parameters, for which the conditions of positivity are different. This always happens if different measures work for   $|q|<1$ and  $|q|>1$ (e.g. compare Al-Salam Carlitz I and II), which in our scheme means that a degeneration is not equal to the degeneration of the flipped version. However it also happens for just simply different parameters while keeping $|q|<1$, for example, consider the relation $(q;q)_nL_n^{(\alpha)}(-q^{-x};q) = C_n(q^{-x};-q^{-\alpha};q)$ between $q$-Laguerre and $q$-Charlier polynomials. 

A special case of this is when we can specialize the product of two parameters to $q^{-N}$ which makes the measure finitely supported, for example compare Askey-Wilson polynomials and $q$-Racah polynomials. These cases will be the same point in our degeneration scheme. However it is relatively easy to determine if we can specialize parameters in this way. This is only possible if we have two visible parameters $\alpha_r$ which sum up to 0 (because their product has to be $q^{-N}$ for all values of $p$).

The degeneration scheme in Figure \ref{figawscheme} shows the different limits we obtain of families of orthogonal polynomials. We modded out by the permutation symmetries in the parameters on the elliptic level, and the translations along $p$-shifts. However we did not mod out by the flip (sending $q\to 1/q$), unlike in the full scheme in Figure \ref{figfulldeg}. We also changed the names from that scheme. The iterated limit property shows that all families are limits of the top level, the Askey-Wilson polynomials. For each degeneration there is a unique associated face which contains $(\alpha;\gamma;\zeta)=v:=(0000;\frac12\frac12;0)$ at its boundary. This vector corresponds to the Askey-Wilson polynomials themselves. In the midpoint of the face denoted by $[a/b]$ we have
$\gamma_0=\gamma_1=a/b$. Moreover the degeneration $[a/b]$ is a $(b-4)/2$ level degeneration down from the Askey-Wilson level. The level of degeneration and the value of $\gamma_0=\gamma_1$ almost completely determine which degeneration we are in, except that there are two limits corresponding to $[5/10]$, so we called one of them $[5/10]'$. 

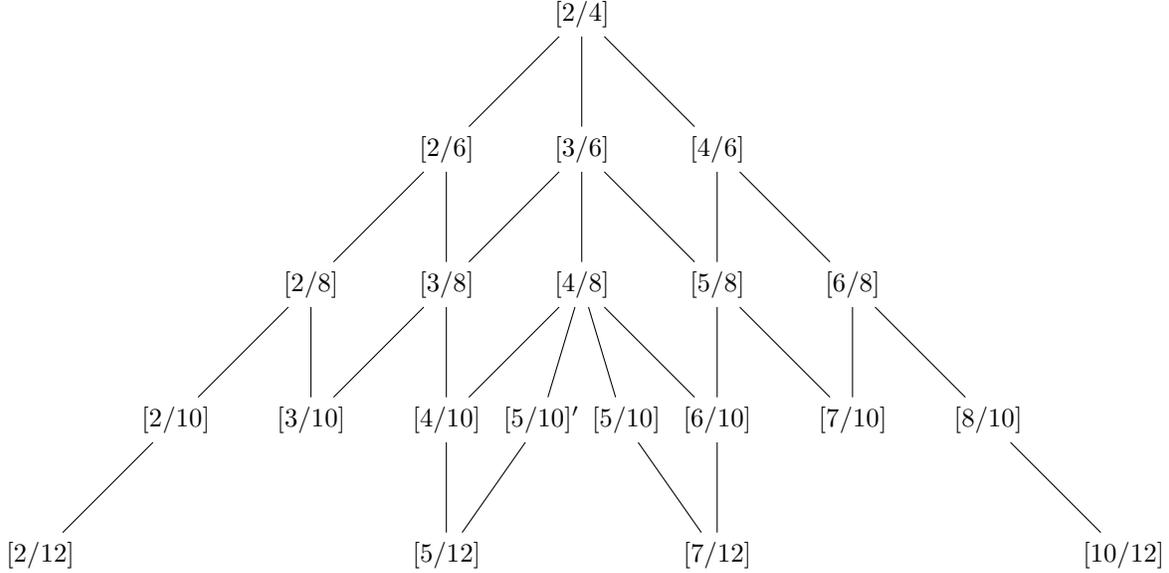
\begin{figure}

\begin{tikzpicture}[scale=0.9]

\node  (24) at (8 ,8)  {$[2/4]$}  ;

\node (26) at (6,6) {$[2/6]$} ;
\node (36) at (8,6) {$[3/6]$} ;
\node (46) at (10,6) {$[4/6]$} ;

\node (28) at (4,4) {$[2/8]$} ;
\node (38) at (6,4) {$[3/8]$} ;
\node (48) at (8,4) {$[4/8]$} ;
\node (58) at (10,4) {$[5/8]$} ;
\node (68) at (12,4) {$[6/8]$} ;

\node (2x) at (2,2) {$[2/10]$} ;
\node (3x) at (4,2) {$[3/10]$} ;
\node (4x) at (6,2) {$[4/10]$} ;
\node (5x) at (7.4,2) {$[5/10]'$} ;  %Check
\node (5xf) at (8.6,2) {$\phantom{'}[5/10]$} ; % Check
\node (6x) at (10,2) {$[6/10]$} ;
\node (7x) at (12,2) {$[7/10]$} ;
\node (8x) at (14,2) {$[8/10]$} ;

\node (2xii) at (0,0) {$[2/12]$} ;
\node (5xii) at (6,0) {$[5/12]$} ;
\node (7xii) at (10,0) {$[7/12]$} ;
\node (10xii) at (16,0) {$[10/12]$} ;

\draw (10xii) -- (8x) -- (68) -- (46)-- (24)-- (26) -- (28) -- (2x) -- (2xii);
\draw (26) -- (38) -- (36) -- (58) -- (46) ;
\draw (24) -- (36) -- (48) -- (5x);
\draw (28) -- (3x) -- (38) -- (4x) -- (48) -- (6x) -- (58) -- (7x) -- (68) ;
\draw (48) -- (5xf);

\draw (4x) -- (5xii) -- (5x) ;
\draw (5xf) -- (7xii) -- (6x) ;

\end{tikzpicture}

\caption{The $q$-Askey scheme as part of the degeneration scheme for biorthogonal $q$-hypergeometric rational functions.}\label{figawscheme}

\end{figure}

We can also put the different degenerations in a list, giving their identification with elements of the usual $q$-Askey scheme. The table lists the names we gave the degeneration (both in our scheme of orthogonal polynomials, and the complete scheme of biorthogonal functions), the midpoint of the face giving this limit which contains $v$,
and the classical polynomials it corresponds to (both for measures with infinite support and, if it exists, with finite support). These families can all be found in \cite{KS}, except for the ones on the right most line (below the top level), and the ones for which no positive measure is possible (denoted by $NP$  in the scheme below). 
The case $[6/8]$, i.e. Al-Salam Chihara polynomials with $|q|>1$, was studied by Askey and Ismail \cite{AI}. 

\begin{center}
\begin{tabular}{lllll}
Fig. \ref{figawscheme} &  Midpoint & Fig. \ref{figfulldeg} & $q$-Askey &  Discrete  \\
$[2/4]$ & $(0,0,0,0;\frac12,\frac12;0)$ & $40as$ & Askey-Wilson &  $q$-Racah \\
\hline
$[2/6]$ & $(0,0,0,\frac13;\frac13,\frac13;0)$ & $31as$ & Continuous dual $q$-Hahn & Dual $q$-Hahn \\
$[3/6]$ & $(-\frac16,-\frac16,\frac16,\frac16;\frac12,\frac12;\frac16)$ & $2200as$ & Big $q$-Jacobi & $q$-Hahn \\
$[4/6]$ & $(-\frac13,0,0,0;\frac23,\frac23;0)$ & $31as$ & ? & ? \\
\hline
$[2/8]$ & $(0,0,\frac14,\frac14;\frac14,\frac14;0)$ & $22as$ & Al-Salam Chihara & Dual $q$-Krawtchouk \\
$[3/8]$ & $(-\frac18,-\frac18,\frac18,\frac38;\frac38,\frac38;\frac18)$ & $2110as$ & Big $q$-Laguerre & affine $q$-Krawtchouk \\
$[4/8]$ & $(-\frac14,0,0,\frac14;\frac12,\frac12;\frac14)$ & $1120as$ & Little $q$-Jacobi & $q$-Krawtchouk \\
$[5/8]$ & $(-\frac38,-\frac18,\frac18,\frac18;\frac58,\frac58;\frac18)$ & $2110as$ & $q$-Meixner & quantum $q$-Krawtchouk \\
$[6/8]$ & $(-\frac14,-\frac14,0,0;\frac34,\frac34;0)$ & $22as$ & Askey-Ismail & ? \\
\hline
$[2/10]$ & $(0,\frac15,\frac15,\frac15;\frac15,\frac15;0)$ & $13as$ & Continuous big $q$-Hermite & - \\
$[3/10]$ & $(-\frac1{10},-\frac1{10},\frac3{10},\frac3{10};\frac3{10},\frac3{10};\frac{1}{10})$ & $2020as$ & 
Al Salam Carlitz I & - \\
$[4/10]$ & $(-\frac15,0,0,\frac2{5};\frac25,\frac25;\frac{1}{5})$ & $1021as$ & Little $q$-Laguerre & - \\
$[5/10]'$ & $(-\frac1{10},-\frac1{10},-\frac1{10},\frac{3}{10};\frac12,\frac12;\frac{3}{10})$ & $1030as$ & NP & - \\   %Check
$[5/10]$ & $(-\frac{3}{10},\frac1{10},\frac1{10},\frac1{10};\frac12,\frac12;\frac{3}{10})$ & $1030as$ & Alternative $q$-Charlier & - \\ % Check
$[6/10]$ & $(-\frac25,0,0,\frac1{5};\frac35,\frac35;\frac{1}{5})$ & $1021as$ & $q$-Laguerre/$q$-Charlier & - \\
$[7/10]$ & $(-\frac3{10},-\frac3{10},\frac1{10},\frac1{10};\frac7{10},\frac7{10};\frac{1}{10})$ & $2020as$ & 
Al Salam Carlitz II & - \\
$[8/10]$ & $(-\frac15,-\frac15,-\frac15,0;\frac45,\frac45;0)$ & $13as$ & ? & - \\
\hline
$[2/12]$ & $(\frac16,\frac16,\frac16,\frac16;\frac16,\frac16;0)$ & $04as$ & Continuous $q$-Hermite & - \\
$[5/12]$ & $(-\frac1{12},-\frac1{12},-\frac1{12},\frac5{12};\frac5{12},\frac5{12};\frac14)$ & $0031as$ & NP & - \\
$[7/12]$ & $(-\frac5{12},\frac1{12},\frac1{12},\frac1{12};\frac7{12},\frac7{12};\frac14)$ & $0031as$ & Stieltjes-Wiegert & - \\
$[10/12]$ & $(-\frac16,-\frac16,-\frac16,-\frac16;\frac56,\frac56;0)$ & $04as$ & ? & - \\
\end{tabular}
\end{center}

\section{Pastro Polynomials}
We would like to draw attention to the special points in the degeneration scheme for which the limits are families of biorthogonal polynomials (i.e. polynomials outside the $q$-Askey scheme where we have orthogonality). These polynomials are the Pastro polynomials \cite{Pastro}, and its degenerations. 

The top level of these polynomials is $1111pp$, associated to the vector $\alpha=(-\frac14,0,\frac14,\frac12;0,\frac12)$ (which equals it's own flip). For the first function 
$\tilde R_{n}(zp^{-\frac14};t_0 p^{-\frac14}, t_1,t_2p^{\frac14},t_3p^{\frac12};u_0,u_1p^{\frac12})$ the valuation turns out to be zero, and we get the explicit expression 
\begin{align}\label{eqdefPn}
P_n(z)& :=  \lim_{p\to 0} \tilde R_{n}(zp^{-\frac14};t_0 p^{-\frac14}, t_1,t_2p^{\frac14},t_3p^{\frac12};u_0,u_1p^{\frac12}) 
 \\ & = {}_3\phi_2\left( \begin{array}{c} q^{-n}, \frac{t_0}{z}, \frac{q}{u_0t_1} \\ 
 t_0t_2,0 \end{array} ;q,q\right)
=   
 \frac{(1/t_3u_1;q)_n}{(t_0t_2;q)_n} \left( \frac{q}{t_1u_0}\right)^n {}_2\phi_1 \left( \begin{array}{c} \frac{q}{t_1u_0}, q^{-n} \\ q^{1-n}t_3u_1\end{array} ;q, \frac{q}{t_2z} \right). \nonumber 
\end{align}
The ${}_3\phi_2$ expression is obtained by interchanging limit and sum in the definition \eqref{equnivbiortho} of the biorthogonal function $\tilde R_n$. The ${}_2\phi_1$ expression then follows using \cite[(III.7)]{GR}.

Notice that the parameters of the function only appear in certain combinations. In particular if we define
%write  $A=\frac{q}{t_1u_0}$, $B=\frac{q}{t_3u_1}$ and $w=\frac{q^{1/2}}{t_2t_3u_1z}$ we can define
\[
p_{n}(w;A,B) = 
\frac{(B/q;q)_n}{(AB/q;q)_n} A^n {}_2\phi_1 \left( \begin{array}{c} A, q^{-n} \\ q^{2-n}/B %Bq^{-n} 
\end{array} ;q, \frac{wq^{3/2}}{B} \right).
\]
we can write
\[
P_{n}(z;t_0:t_1,t_2,t_3;u_0,u_1) = p_n(\frac{q^{1/2}}{t_2t_3u_1z };\frac{q}{t_1u_0},\frac{q}{t_3u_1}).
\]
For the right hand family (the functions with $u_0$ and $u_1$ interchanged) we notice that by the symmetries from Lemma \ref{lemRtildeabeluniv}, and the permutation symmetry in the $t_r$'s we have 
\begin{align*}
\tilde R_n(zp^{-\frac14};&t_0 p^{-\frac14}, t_1,t_2p^{\frac14},t_3p^{\frac12};u_1p^{\frac12},u_0)
= \tilde R_n(zp^{\frac14};t_0 p^{\frac14}, t_1p^{\frac12},t_2p^{-\frac14},t_3;u_1,u_0p^{\frac12})
\\& = \tilde R_n(z^{-1}p^{-\frac14};t_0 p^{\frac14}, t_1p^{\frac12},t_2p^{-\frac14},t_3;u_1,u_0p^{\frac12})
\\ &= \frac{\theta(p^{-\frac14} t_2t_3, p^{\frac14} t_1t_2, 
p^{\frac54} \frac{qt_0}{u_1} , p^{-\frac14} \frac{1}{t_2u_0})_n}
{\theta(p^{\frac34} t_0t_1, p^{\frac14} t_0t_3, p^{-\frac34} \frac{1}{t_0u_0}, p^{\frac34} \frac{qt_2}{u_1})_n}
\tilde R_n(z^{-1}p^{-\frac14};t_2 p^{-\frac14}, t_3,t_0p^{\frac14},t_1p^{\frac12};u_1,u_0p^{\frac12}).
\end{align*}
As a corollary we see that the limit on the right hand side is up to a constant equal to the limit on the left hand side with different parameters. Indeed we have
\begin{align*}
Q_n &:= \lim_{p\to 0} \tilde R_n(zp^{-\frac14};t_0 p^{-\frac14}, t_1,t_2p^{\frac14},t_3p^{\frac12};u_1p^{\frac12},u_0)  %\\
%&
=  \left( \frac{1}{u_0t_0t_1t_2}\right)^{n}  P_n(\frac{1}{z};t_2,t_3,t_0,t_1;u_1,u_0).
\end{align*}
In particular, setting
\[
q_n(w;A,B) := p_n(\frac{1}{w};B,A),
\]
we have 
\[
Q_n(z;t_0:t_1,t_2,t_3;u_0,u_1) = 
\left( \frac{q}{t_3u_1}\right)^{-n}
q_n(\frac{q^{1/2}}{t_2t_3u_1z };\frac{q}{t_1u_0},\frac{q}{t_3u_1}).
\]
Note that this is the same parameter correspondence as between $P_{n}$ and $p_n$.
 Direct comparison now shows that the univariate functions $p_{n}$ and $q_{n}$ are equal to the polynomials \cite[(3.1)]{Pastro}, up to a constant and under the parameter correspondence $q^{\alpha}=A$, $q^{\beta} =B$ and $t=w$.

The limit of the biorthogonality relation gives us
\begin{align*}
\langle P_n(\cdot; t_r;u_r;q), Q_m(\cdot;t_r;u_r;q) \rangle = 
\delta_{n,m}  (t_1u_0)^{-n} 
\frac{(q;q)_n}{(t_0t_2;q)_n}.
\end{align*}
where 
\[
\langle f,g\rangle =  \frac{(q, t_0t_2;q)}{(q/t_1u_0,q/t_3u_1;q)} 
\int_{C} f(z)g(z) \frac{\theta(qz/t_0t_1u_0;q)}{(t_0/z,t_2z;q)} \frac{dz}{2\pi i z},
\]
or in terms of $p_n$ and $q_n$ this becomes
\[
\langle p_n(\cdot;A,B), q_m(\cdot;A,B) \rangle = \delta_{n,m} 
\left( \frac{AB}{q}\right)^{n} \frac{(q;q)_n}{(\frac{AB}{q};q)_n},
\]
where
\[
\langle f,g\rangle = \frac{(q, \frac{AB}{q};q)}{(A,B;q)} 
\int_{C} f(w)g(w) \frac{\theta(q^{1/2}w ;q)}{(Aw/q^{1/2} ,B/wq^{1/2};q)} \frac{dw}{2\pi i w}.
\]
%In the univariate case this inner product relation reduces indeed to the biorthogonality \cite[(3.2)]{Pastro} of the Pastro polynomials%\footnote{Pastro has a necessary condition on the parameters which we lack, the difference is that he insists that the contour is the unit circle, while we look at generalized contours (which are allowed to make detours to include points outside the unit circle and exclude points inside the unit circle)}
. 

We would like to consider a special case. Indeed we want to specialize $t_3u_1\to 1$, or equivalently $B\to q$. The simplest expression to do this in is in the original ${}_3\phi_2$ of \eqref{eqdefPn}, which then becomes a ${}_2\phi_1$ which can be summed by the $q$-Vandermonde \cite[(II.6)]{GR}, to give
\[
p_n(w;A,q) = w^n A^n q^{-n/2},
\]
that is, the polynomials $p_n$ become just the monomials $w^n$. 
Observe that the specialization $A,B\to q$ turns the limiting measure into the trivial $\frac{dw}{2\pi i w}$, and indeed we know that the orthogonal functions with respect to this measure are $w^n$ and $w^{-n}$. The multivariate analogues of this special case of the Pastro polynomials are the Macdonald polynomials \cite{Macdonald} (indeed the univariate part of the measure associated to the orthogonality of the Macdonald polynomials is the trivial $dw/2\pi i w$).

\appendix

\section{Description of the limiting systems}\label{secsystemlimit}
In this section we write down explicitly all the limiting families of biorthogonal systems, and try and give a universal picture of the different cases. After the theory of Section \ref{seclimsys} this has reduced to a concrete combinatorial problem.

In Section \ref{seclimsys} we have seen that the limits of pairs of biorthogonal functions, together with measure, are classified by the different kinds of faces of the splitting of polytope $P^{(0)}$ in $P_I$, $P_{II}$ and $P_{III}$ (as in \cite{vdBR}) (with $\zeta$ given by its prescribed value), with the exception of the interior of $P_{II}$, modulo the symmetry action $S_4\times S_2$. Moreover we also still want to mod out by the shifting symmetries (while we are looking in a fundamental domain for this shift-action, boundary points can of course still be mapped to other boundary points). On the other hand we consider two solutions associated to each other by the flip as different (as the measures only make sense if we impose $|q|<1$, while the flip maps $q\to 1/q$, thus making it impossible for this condition to be satisfied for both cases).

The different kind of limits are named according to the relation of the $\alpha$ to $\zeta$. If $\zeta=0$ or $\zeta=1/2$ we have two integers in front, the first one counting how many $\alpha_r$ ($0\leq r\leq 3$) are equal to $\zeta$ (or some integer shift of $\zeta$), the second counting how many are unequal to $\zeta$ (and its integer shifts). If $\zeta\neq 0$ and $\zeta\neq 1/2$, we have four integers. The first two counting how many $\alpha_r$ are equal to $\zeta$ (and its integer shifts), respectively $-\zeta$ (and its integer shifts). The second two counting how many are in the interval $(-\zeta,\zeta)$ (and shifts), respectively $(\zeta,1-\zeta)$ (and shifts). These two pairs of integers are sorted large to small. The letters at the end describe the relation of $\gamma_0=\alpha_4$ and $\gamma_1=\alpha_5$ to $\zeta$. If (modulo 1) $\alpha_4=\alpha_5=\pm \zeta$ then we write $v2$, if 
$\alpha_4=-\alpha_5=\pm \zeta$ we write $vv$, if $\alpha_4=\pm \zeta$ and $\alpha_5 \neq \pm \zeta$ or vice versa, we write $vp$, if $\alpha_4, \alpha_5 \in (-\zeta,\zeta)$ or 
$\alpha_4,\alpha_5 \in (\zeta,1-\zeta)$ (and shifts), we write $as$, and if $\alpha_4\in (-\zeta,\zeta)$ and $\alpha_5 \in (\zeta,1-\zeta)$ or vice versa we write $pp$.

These names give a rough description of the associated biorthogonal functions. Indeed if $\alpha_r=\pm \zeta$ ($0\leq r\leq 3$) then we will be able to ``see'' this parameter in the limit, otherwise the limiting function will not depend on it. If two $\alpha_r$'s are equal the limit will moreover be symmetric under their interchange. If $\alpha_4=\pm \zeta$ or $\alpha_5=\pm \zeta$ then the associated biorthogonal functions will be rational functions, otherwise they will be polynomials. Again if they are equal or in the same interval, as in the cases $v2$ and $as$, the two families will be related by interchanging $u_0$ and $u_1$. The $pp$ limits are Pastro Polynomials and their limits, while the $as$ limits are limits in the Askey-Scheme of $q$-hypergeometric orthogonal polynomials. 

The simplest way to arrive at this classification is to start with the different faces mod $S_6$ of $P_I$, $P_{II}$ and $P_{III}$ from \cite{vdBR}, and tally their $S_4\times S_2$ orbits. Subsequently we identify those faces which differ by a shift. The flip symmetry is treated separately, as it involves inverting $q$ and thus alters the feel of the functions. All faces are simplicial, so we will just give the different vertices of the faces, together with the midpoints. Finally we briefly indicate the type of measure, whether it is a Nasrallah-Rahman type beta integral (NR) (Proposition \ref{propPI}), a symmetry broken integral (SB) (Proposition \ref{proplimip3}), a single series measure $\Sigma$ (Proposition \ref{propsumlim}), or a double series measure ($\Sigma^2$) (Proposition \ref{proplimip2}, these double series can also have an expression as a symmetry broken integral). 
%In case multiple expressions work (this only happens for double sums and symmetry broken integrals, we opted for saying it was a double sum.% (as its dual expressions as series and integral lend itself easier to taking limits).

We will use the following notations for the vertices: $d_j$ ($0\leq j\leq 3$) is the vertex with all 0's except $\alpha_j=1$ (so $|\zeta+\frac12|=\frac12$). $e_j$ ($j=0,1$) is the vertex with all 0's except $\gamma_j=1$ (with again $|\zeta+\frac12|=\frac12$). 
$f_{ij}$ is the vector with all $\frac12$'s, except $\alpha_i=\alpha_j=-\frac12$ (and $|\zeta+\frac12|=0$). $g_{ij}$ is the vector with all $\frac12$'s, except $\alpha_i=\gamma_j=-\frac12$. Finally $h_{01}=(\frac12,\frac12,\frac12,\frac12;-\frac12,-\frac12;0)$.

\subsection{Top level}
\subsubsection{$40v2$}
At the top level we consider vertices of $P_{I}$, $P_{II}$, and $P_{III}$ and get up to $S_4\times S_2$ symmetry the vectors $d_3$, $e_1$, $f_{01}$, $g_{00}$ and $h_{01}$.
These vectors are all mapped to each other by purely shifts, so in fact we only have one case. However the different vertices do correspond to different measures with respect to which the limiting rational functions are biorthogonal.

\subsection{Level 2}
Let us tabulate the different edges up to the $S_4\times S_2$ symmetry. 
Then we group those which are related to each other by shifts in the $W(E_6)$ lattice. The midpoints of the faces are given as the vector $(\alpha_0,\alpha_1,\alpha_2,\alpha_3;\gamma_0,\gamma_1;-\zeta)$. (We negated $\zeta$ as in the polytope $\zeta$ is always positive, but by the $z\to 1/z$ symmetry of the elliptic hypergeometric biorthogonal functions this negative value of $\zeta$ is just a consequence of some choices we made).
\begin{center}
\begin{tabular}{c||c|c|c}
Name & Vertices & Midpoint & $\mu$  \\
\hline
$31vp$ \T \B& $e_1, d_3$ & $(0,0,0,\frac12;0,\frac12;0)$ & NR\\
\B & $f_{01}, g_{00}$ &$(-\frac12,0,\frac12,\frac12;0,\frac12;\frac12)$ & $\Sigma$ \\
 \B& $g_{00}, h_{01}$ & $(0, \frac12,\frac12,\frac12;-\frac12,0;\frac12)$ & $\Sigma$ \\
\hline
$2200vp$ \T \B& $e_1, f_{01}$ & $(-\frac14,-\frac14,\frac14,\frac14;\frac14,\frac34;\frac14)$ & $\Sigma^2$ \\
 \B& $d_3, g_{00}$ & $(-\frac14,\frac14,\frac14,\frac34;-\frac14,\frac14;\frac14)$ & $\Sigma^2$ \\
\hline
$3100v2$ \T \B& $e_1, g_{00}$ & $(-\frac14,\frac14,\frac14,\frac14;-\frac14,\frac34;\frac14)$ & $\Sigma^2$\\
 \B& $d_3, f_{01}$ & $(-\frac14,-\frac14,\frac14,\frac34;\frac14,\frac14;\frac14)$& $\Sigma^2$ \\
 \B& $d_3, h_{01}$ & $(\frac14,\frac14,\frac14,\frac34;-\frac14,-\frac14;\frac14)$ & $\Sigma^2$\\
\hline
$22v2$ \T \B& $d_2, d_3$ & $(0,0,\frac12,\frac12;0,0;0)$ & NR\\
 \B &  $f_{01}, f_{02}$ & $(-\frac12,0,0,\frac12;\frac12,\frac12;\frac12)$ & $\Sigma$ \\
\hline
$40as$  \B& $e_0, e_1$ & $(0,0,0,0;\frac12,\frac12;0)$ & NR\\
 & $g_{00}, g_{01}$ &$(-\frac12,\frac12,\frac12,\frac12;0,0;\frac12)$ & $\Sigma$  
\end{tabular}
\end{center}
The case $40as$ is the top level case in the $q$-Askey-scheme and corresponds to Koornwinder polynomials, or univariately, Askey-Wilson polynomials. We obtain two different measures, the one associated to the edge between $e_0$ and $e_1$ is the usual Askey-Wilson measure.

\subsection{Level 3}
Now we consider 2d faces, that is, triangles. In this case we have some faces which map to different faces if we perform the 
$q\to 1/q$ flip. Considering they are essentially the same functions, but also realizing that we consider them on different domains
(either $|q|<1$ or $|q|>1$) we group them together but are aware of the differences.
\begin{center}
\begin{tabular}{c||c|c|c||c|c|c}
\multirow{2}{*}{Name} & \multicolumn{3}{c||}{Ordinary}  & \multicolumn{3}{c}{Flipped}  \\
& vertices & midpoint & $\mu$& vertices & midpoint & $\mu$ \\
\hline
$22vp$ \T \B  & $d_2, d_3, e_1$ & $(0,0,\frac13,\frac13;0,\frac13;0)$ & NR
 & $f_{01}, f_{02}, g_{00}$ & $(-\frac12,\frac16,\frac16,\frac12;\frac16,\frac12;\frac12)$ & $\Sigma$ \\
\B& $f_{01}, g_{00}, g_{10}$ & $(-\frac16,-\frac16,\frac12,\frac12;-\frac16,\frac12;\frac12)$ & SB
  & $g_{00}, g_{10}, h_{01}$ & $(\frac16,\frac16,\frac12,\frac12;-\frac12,\frac16;\frac12)$ & $\Sigma$ \\
  \hline
$2110vp$ \T \B& $d_3, e_1, f_{01}$ & $(-\frac16,-\frac16,\frac16,\frac12;\frac16,\frac12;\frac16)$ & $\Sigma^2$ 
 & $d_3, e_1, g_{00}$ & $(-\frac16,\frac16,\frac16,\frac12;-\frac16,\frac12;\frac16)$ & $\Sigma^2$ \\
 \B& $d_3, g_{00}, h_{01}$ & $(0,\frac13,\frac13,\frac23;-\frac13,0;\frac13)$ & $\Sigma$ 
  & $d_3, f_{01}, g_{00}$ & $(-\frac13,0,\frac13,\frac23;0,\frac13;\frac13)$ & $\Sigma$ \\
 \B& $e_1, g_{00}, f_{01}$ & $(-\frac13,0,\frac13,\frac13;0,\frac23;\frac13)$ & $\Sigma$ & & &\\
\hline  
$1120vv$ \T \B& $d_2, d_3, g_{00}$ & $(-\frac16,\frac16,\frac12,\frac12;-\frac16,\frac16;\frac16)$ & $\Sigma^2$ &&& \\
  & $d_3, g_{00}, g_{10}$ & $(0,0,\frac13,\frac23;-\frac13,\frac13;\frac13)$ & $\Sigma$ && & \\
  \B& $e_1, f_{01}, f_{02}$ & $(-\frac13,0,0,\frac13;\frac13,\frac23;\frac13)$ & $\Sigma$  &&& \\
  \hline
 $2020v2$\T \B & $d_2, d_3, f_{01}$ & $(-\frac16,-\frac16,\frac12,\frac12;\frac16,\frac16;\frac16)$ &  $\Sigma^2$
 &$d_2, d_3, h_{01}$ & $(\frac16,\frac16,\frac12,\frac12;-\frac16,-\frac16;\frac16)$ & $\Sigma^2$\\
 \B& $e_1, g_{00}, g_{10}$ & $(0,0,\frac13,\frac13;-\frac13,\frac23;\frac13)$ & $\Sigma$
  & $d_3, f_{01}, f_{02}$ & $(-\frac13,0,0,\frac23;\frac13,\frac13;\frac13)$ & $\Sigma$ \\
  \hline
$13v2$\T \B & $d_1, d_2, d_3$ & $(0,\frac13,\frac13,\frac13;0,0;0)$ & NR
  & $f_{01}, f_{02}, f_{03}$ & $(-\frac12,\frac16,\frac16,\frac16;\frac12,\frac12;\frac12)$ & $\Sigma$ \\
  \B& $f_{01}, f_{02}, f_{12}$ & $(-\frac16,-\frac16,-\frac16,\frac12;\frac12,\frac12;\frac12)$ & SB 
  & $g_{00}, g_{10}, g_{20}$ & $(\frac16,\frac16,\frac16,\frac12;-\frac12,\frac12;\frac12)$ & $\Sigma$ \\
  \hline
 $31as$ \T\B & $d_3, e_1, e_0$ & $(0,0,0,\frac13;\frac13,\frac13;0)$ & NR
  & $f_{01}, g_{00}, g_{01}$ & $(-\frac12,\frac16,\frac12,\frac12;\frac16,\frac16;\frac12)$ & $\Sigma$ \\
  \B& $g_{00}, g_{01}, h_{01}$ & $(-\frac16,\frac12,\frac12,\frac12;-\frac16,-\frac16;\frac12)$ & SB &&&\\
 \hline
 $2200as$\T \B & $d_3, g_{00}, g_{01}$  &  $(-\frac13,\frac13,\frac13,\frac23;0,0;\frac13)$ & $\Sigma$   &&&\\
  & $e_0, e_1, f_{01}$ & $(-\frac16,-\frac16,\frac16,\frac16;\frac12,\frac12;\frac16)$ & $\Sigma^2$ && &
\end{tabular}
\end{center}

%The $2200as$ case are the Big $q$-Jacobi polynomials, while the $31as$ case are continuous $q$-Hahn polynomials (and the flip gives  continuous $q^{-1}$-Hahn polynomials). 

The points where we do not include flipped versions, in this table or the coming ones, are all their own flips. In particular for those points the same pair of families of functions is biorthogonal with respect to explicit measures both when $|q|>1$ and when $|q|<1$.

\subsection{Level 4} The 3-d faces are all tetrahedra.

\begin{center}
\begin{tabular}{c||c|c|c||c|c|c}
\multirow{2}{*}{Name} & \multicolumn{3}{c||}{Ordinary}  & \multicolumn{3}{c}{Flipped}  \\
& vertices & midpoint & $\mu$ & vertices & midpoint & $\mu$ \\
\hline
$13vp$ \T \B& $d_1, d_2, d_3,  e_1$ &  $(0,\frac14,\frac14,\frac14;0,\frac14;0)$ & NR &
 $f_{01}$, $f_{02}$, $f_{03}$,  $g_{00}$ & $(-\frac12,\frac14,\frac14,\frac14;\frac14,\frac12;\frac12)$ & $\Sigma$ \\
 \B& & & & $g_{00}, g_{10}, g_{20}, h_{01}$ & $(\frac14,\frac14,\frac14,\frac12;-\frac12,\frac14;\frac12)$ &  $\Sigma$ \\
\hline
$2020vp$ \T\B& $d_2, d_3, e_1, f_{01}$ & $(-\frac18,-\frac18,\frac38,\frac38;\frac18,\frac38;\frac18) $ & $\Sigma^2$ &
$d_3, f_{01}, f_{02},  g_{00}$ & $(-\frac38, \frac18,\frac18,\frac58;\frac18,\frac38;\frac38)$ & $\Sigma$ \\
 \B& $e_1, f_{01}, g_{00}, g_{10}$ & $(-\frac18,-\frac18,\frac38,\frac38;-\frac18,\frac58;\frac38)$ & SB & & &\\
\hline
$1021vp$\T\B & $d_2, d_3, f_{01}, g_{00}$ & $(-\frac14,0,\frac12,\frac12;0,\frac14;\frac14)$ & $\Sigma$ &
 $d_2, d_3, g_{00},  h_{01}$ & $(0,\frac14,\frac12,\frac12;-\frac14,0;\frac14)$ & $\Sigma$ \\
 \B& $d_3, e_1, g_{00},  g_{10}$ & $(0,0,\frac14,\frac12;-\frac14,\frac12;\frac14) $ & $\Sigma$ &
  $d_3, e_1, f_{01},  f_{02}$ & $(-\frac14,0,0,\frac12;\frac14,\frac12;\frac14)$ & $\Sigma$ \\
\hline
$1030vv$\T \B& $d_1, d_2, d_3, g_{00}$ & $(-\frac18,\frac38,\frac38,\frac38;-\frac18,\frac18;\frac18)$ & $\Sigma^2$
 & $d_3, g_{00}, g_{10}, g_{20}$ & $(\frac18,\frac18,\frac18,\frac58;-\frac38,\frac38;\frac38)$ &  $\Sigma$  \\
  \B& $e_1, f_{01}, f_{02}, f_{12}$ & $(-\frac18,-\frac18,-\frac18,\frac38;\frac38,\frac58;\frac38)$ & SB
 & $e_1, f_{01}, f_{02}, f_{03}$ & $(-\frac38,\frac18,\frac18,\frac18;\frac38,\frac58;\frac38)$ & $\Sigma$ \\
\hline
$1111pp$ \T\B& $d_3, e_1, f_{01}, g_{00}$ & $(-\frac14,0,\frac14,\frac12;0,\frac12;\frac14)$ & $\Sigma$ & &&\\
\hline
$1120vp$ \T\B& $d_2, d_3, e_1, g_{00}$ & $(-\frac18,\frac18,\frac38,\frac38;-\frac18,\frac38;\frac18)$ & $\Sigma^2$ 
  & $d_3, g_{00}, g_{10}, h_{01}$ & $(\frac18,\frac18, \frac38,\frac58;-\frac38,\frac18;\frac38)$ &  $\Sigma$  \\
 \B& $d_3, f_{01}, g_{00},  g_{10}$ & $(-\frac18,-\frac18, \frac38,\frac58;-\frac18,\frac38;\frac38)$ & SB
  & $e_1, f_{01}, f_{02}, g_{00}$ & $(-\frac38,\frac18,\frac18,\frac38;\frac18,\frac58;\frac38)$ &  $\Sigma$ \\
\hline
$0022vv$ \T\B& $d_2, d_3, g_{00}, g_{10}$ & $(0,0,\frac12,\frac12;-\frac14,\frac14;\frac14)$ & $\Sigma$  & &&\\
\hline
$1030v2$ \T\B& $d_1, d_2, d_3, h_{01}$ & $(\frac18,\frac38,\frac38,\frac38;-\frac18,-\frac18;\frac18)$ & $\Sigma^2$
  & $e_1, g_{00}, g_{10}, g_{20}$ & $(\frac18,\frac18,\frac18,\frac38;-\frac38,\frac58;\frac38)$ &  $\Sigma$ \\
 \B& $d_3, f_{01}, f_{02}, f_{12}$ & $(-\frac18,-\frac18,-\frac18,\frac58;\frac38,\frac38;\frac38)$ & SB & &&\\
\hline
$04v2$ \T\B& $d_0, d_1, d_2,d_3$ & $(\frac14,\frac14,\frac14,\frac14;0,0;0)$ & NR
 & $g_{00}, g_{10}, g_{20}, g_{30}$ & $(\frac14,\frac14,\frac14,\frac14;-\frac12,\frac12;\frac12)$ &  $\Sigma$ \\
\hline
$22as$ \T\B& $d_2, d_3, e_0, e_1$ & $(0,0,\frac14,\frac14;\frac14,\frac14;0)$ & NR
 & $f_{01}, f_{02}, g_{00}, g_{01}$ & $(-\frac12,\frac14,\frac14,\frac12;\frac14,\frac14;\frac12)$ &  $\Sigma$ \\
\hline
$2110as$\T \B& $d_3, e_0, e_1, f_{01}$ & $(-\frac18,-\frac18,\frac18,\frac38;\frac38,\frac38;\frac18)$ & $\Sigma^2$
 & $d_3, f_{01}, g_{00},  g_{01}$ & $(-\frac38,\frac18,\frac38,\frac58;\frac18,\frac18;\frac38)$ &  $\Sigma$  \\
 \B& $d_3, g_{00}, g_{01},  h_{01}$ & $(-\frac18,\frac38,\frac38, \frac58;-\frac18,-\frac18;\frac38)$ & SB & &&\\
\hline
$1120as$ \T\B& $d_2, d_3, g_{00}, g_{01}$ & $(-\frac14,\frac14,\frac12,\frac12;0,0;\frac14)$ & $\Sigma$ & &&\\ 
 & $e_0, e_1, f_{01}, f_{02}$  & $(-\frac14,0,0,\frac14;\frac12,\frac12;\frac14)$ & $\Sigma$ & &&
\end{tabular}

\end{center}

\subsection{Level 5} The 4-d faces are all simplices as well. To save some horizontal space we use an abbreviated notation for the midpoints. For example $(\frac15^4;0,\frac15;0)$ denotes the point $(\frac15,\frac15,\frac15,\frac15;0,\frac15;0)$. 
\begin{center}
\begin{tabular}{c@{\ }||@{\ }c@{\,}|@{\,}c@{\,}|@{\,}c@{\,}||@{\ }c@{\,}|@{\,}c@{\,}|@{\,}c}
\multirow{2}{*}{Name} & \multicolumn{3}{c@{\,}||@{\ }}{Ordinary}  & \multicolumn{3}{c}{Flipped}  \\
& vertices & midpoint &$\mu$& vertices & midpoint & $\mu$ \\
\hline
$04vp$ \Ttwo \B& $d_0, d_1, d_2, d_3,  e_1$ & $(\frac15^4;0,\frac15;0)$ & NR &
 $g_{00}, g_{10}, g_{20}, g_{30}, h_{01}$ & $(\frac3{10}^4;-\frac12,\frac3{10};\frac12)$ & $\Sigma$\\
\hline
$0031vp$ \Ttwo& $d_1, d_2, d_3, g_{00},  h_{01}$ & $(0,\frac25^3;-\frac15,0;\frac15)$ & SB
 & $d_3, e_1, g_{00}, g_{10},  g_{20}$ & $(\frac1{10}^3,\frac12;-\frac3{10},\frac12;\frac3{10})$& $\Sigma$ \\
  \T \B & $d_3, e_1, f_{01}, f_{02}, f_{12}$ & $(-\frac1{10}^3,\frac12;\frac3{10},\frac12;\frac3{10})$& SB & & &\\
\hline
$1021pp$ \Ttwo& $d_2, d_3, e_1, f_{01},  g_{00}$ & $(-\frac15,0,\frac25^2;0,\frac25;\frac15)$ &  SB
  & $e_1, d_3, f_{01}, f_{02},  g_{00}$ & $(-\frac3{10},\frac1{10}^2,\frac12;\frac1{10},\frac12;\frac3{10})$& $\Sigma$ \\
  \B & $d_3, e_1, f_{01}, g_{00},  g_{10}$ & $(-\frac1{10}^2,\frac3{10},\frac12;-\frac1{10},\frac12;\frac3{10})$ & SB & & &\\
\hline
$0022vp$\Ttwo & $d_2, d_3, e_1, g_{00}, g_{10}$ & $(0^2,\frac25^2;-\frac15,\frac25;\frac15)$ & SB
 & $d_2, d_3, g_{00}, g_{10},  h_{01}$ & $(\frac1{10}^2,\frac12^2;-\frac3{10},\frac1{10};\frac3{10})$ & $\Sigma$ \\
 \B & $d_2, d_3, f_{01}, g_{00}, g_{10}$ & $(-\frac1{10}^2,\frac12^2;-\frac1{10}, \frac3{10};\frac3{10})$ & SB  && \\
\hline 
$0040v2$\Ttwo \B& $d_0, d_1, d_2, d_3,  h_{01}$  & $(\frac3{10}^4;-\frac1{10}^2;\frac1{10})$ & $\Sigma^2$
 & $e_1, g_{00}, g_{10}, g_{20},  g_{30}$ &  $(\frac15^4;-\frac25,\frac35;\frac25)$& $\Sigma$  \\
\hline
$1030vp$ \Ttwo& $d_1, d_2, d_3, e_1, g_{00}$ & $(-\frac1{10},\frac3{10}^3;-\frac1{10},\frac3{10};\frac1{10})$ & $\Sigma^2$
 & $d_3, g_{00}, g_{10}, g_{20}, h_{01}$ & $(\frac15^3,\frac35;-\frac25,\frac15;\frac25)$ & $\Sigma$  \\
 \B & & & &$e_1, f_{01}, f_{02}, f_{03},  g_{00}$ & $(-\frac25,\frac15^3;\frac15,\frac35;\frac25)$ & $\Sigma$ \\
 \hline
$13as$ \Ttwo \B& $d_1, d_2, d_3, e_0, e_1$ & $(0,\frac15^3;\frac15^2;0)$ & NR
 & $f_{01}, f_{02}, f_{03}, g_{00},  g_{01}$ & $(-\frac12,\frac3{10}^3;\frac3{10}^2;\frac12)$ &  $\Sigma$ \\
\hline
$2020as$ \Ttwo \B & $d_2, d_3, e_0, e_1,  f_{01}$ & $(-\frac1{10}^2,\frac3{10}^2;\frac3{10}^2;\frac1{10}) $ & $\Sigma^2$
 & $d_3, f_{01}, f_{02}, g_{00}, g_{01}$ & $(-\frac25,\frac15^2,\frac35;\frac15^2;\frac25)$& $\Sigma$ \\
\hline
$1021as$\Ttwo \B & $d_3, e_0, e_1, f_{01},  f_{02}$ & $(-\frac15,0^2,\frac25;\frac25^2;\frac15)$ & SB
 & $d_2, d_3, f_{01}, g_{00},  g_{01}$ & $(-\frac3{10},\frac1{10},\frac12^2;\frac1{10}^2;\frac3{10})$& $\Sigma$  \\
 \B & $d_2, d_3, g_{00}, g_{01},  h_{01}$ & $(-\frac1{10},\frac3{10}, \frac12^2;-\frac1{10}^2;\frac3{10})$& SB  & && \\
\hline
$1030as$\Ttwo \B& $d_1, d_2, d_3, g_{00},  g_{01}$ &  $(-\frac15,\frac25^3;0^2;\frac15)$ & SB
   & $e_0, e_1, f_{01}, f_{02},  f_{03}$ & $(-\frac3{10},\frac1{10}^3;\frac12^2;\frac3{10})$& $\Sigma$ \\
 & $e_0, e_1, f_{01}, f_{02},  f_{12}$ & $(-\frac1{10}^3,\frac3{10};\frac12^2;\frac3{10})$ & SB  && & \\
 \end{tabular}

\end{center}

\subsection{Level 6} The 5-d faces are not all simplices anymore, but the only non-simplices are the interiors of $P_{II,t}$, which are excluded anyway. 
%It should be noted that these are the only faces for which no translate of the flip is in $P^{(0)}$. Thus the section on measures does not give us a measure for the flipped versions of the biorthogonal systems below (i.e. we do not have a measure for $|q|>1$).
\begin{center}
\begin{tabular}{c||c|c|c}
Name & vertices & midpoint & $\mu$ \\
\hline
$0022pp$ \T \B& $d_2, d_3, e_1, f_{01}, g_{00}, g_{10}$&  $(-\frac1{12},-\frac1{12},\frac5{12},\frac5{12};-\frac1{12},\frac5{12};\frac14)$ & SB \\
\hline
$04as$ \T \B & $d_0, d_1, d_2, d_3, e_0, e_1$ & $(\frac16,\frac16,\frac16,\frac16;\frac16,\frac16;0)$ & NR\\
\hline
$0031as$ \T \B& $d_1, d_2, d_3, g_{00}, g_{01}, h_{01}$  & $(-\frac1{12},\frac5{12},\frac5{12},\frac5{12};-\frac1{12},-\frac1{12};\frac14)$  & SB\\
  & $d_3, e_0, e_1, f_{01}, f_{02}, f_{12}$ & $(-\frac1{12},-\frac1{12},-\frac1{12},\frac5{12};\frac5{12},\frac5{12};\frac14)$ & SB
\end{tabular}
\end{center}

In these cases the flips of the points are not in the polytope $P^{(0)}$ (thus in particular these points are not invariant under the flip). This implies that for the pairs of families of polynomials under consideration the measures 
from \cite{vdBR} only work when $|q|<1$, not when $|q|>1$. In \cite{vdBRmeas} we show that we can give measures in these cases also for $|q|>1$, the measures then being related to bilateral series.

\subsection{Degeneration scheme}
Having tabulated all these different families of biorthogonal functions it becomes possible to draw them all in a single graph, see Figure \ref{figdegscheme}. 

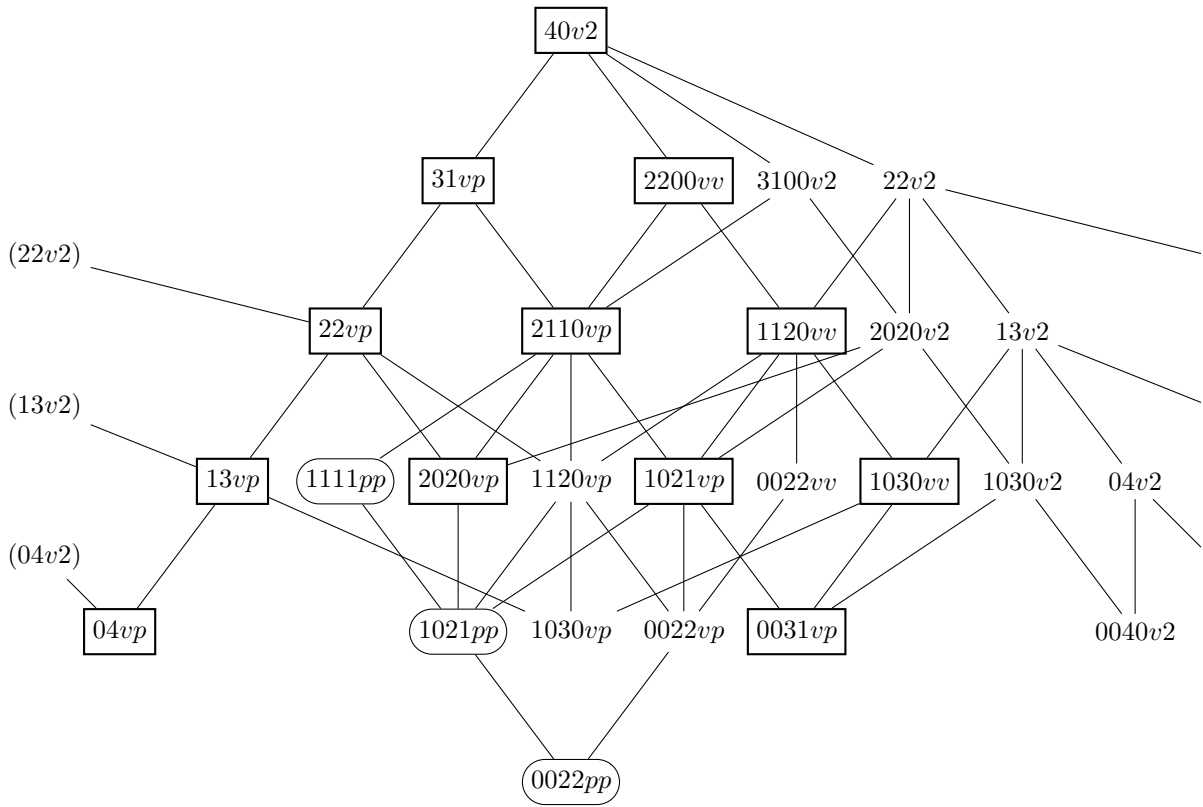
\begin{figure}
\begin{tikzpicture}[pastro/.style={rectangle, minimum size=6mm, rounded corners=3mm, draw=black}, aw/.style={rectangle, minimum size=6mm, thick,  draw=black}]

\node  (0022pp) at (7,0) [pastro] {$0022pp$}  ;

\node  (04vp) at (1,2) [aw]  {$04vp$}  ;
\node (1021pp) at (5.5,2) [pastro]  {$1021pp$};
\node (0022vp) at (8.5,2)  {$0022vp$};
\node (0031vp) at (10,2) [aw]  {$0031vp$};
\node (0040v2)  at (14.5,2) {$0040v2$};
\node (1030vp) at (7,2)  {$1030vp$};

\node (13vp) at (2.5,4) [aw] {$13vp$};
\node (1111pp) at (4,4) [pastro] {$1111pp$}; 
\node (2020vp) at (5.5,4) [aw] {$2020vp$};
\node (1120vp) at (7,4) {$1120vp$};
\node (1021vp) at (8.5,4) [aw] {$1021vp$};
\node (0022vv) at (10,4) {$0022vv$};
\node (1030vv) at (11.5,4) [aw] {$1030vv$};
\node (1030v2) at (13,4) {$1030v2$};
\node (04v2) at (14.5,4) {$04v2$};
\node (d04v2) at (0,3) {$(04v2)$};
\node (t04v2) at (15.5,3) {} ;

\node (22vp) at (4,6) [aw] {$22vp$};
\node (2110vp) at (7,6) [aw] {$2110vp$};
\node (1120vv) at (10,6) [aw] {$1120vv$};
\node (2020v2) at (11.5,6) {$2020v2$};
\node (13v2) at (13,6) {$13v2$};
\node (d13v2) at (0,5) {$(13v2)$};
\node (t13v2) at (15.5,5) {} ;

\node (31vp) at (5.5,8) [aw] {$31vp$};
\node (2200vv) at (8.5,8) [aw] {$2200vv$};
\node (3100v2) at (10,8) {$3100v2$};
\node (22v2) at (11.5,8) {$22v2$} ;
\node (d22v2) at (0,7) {$(22v2)$};
\node (t22v2) at (15.5,7) {} ;

\node (40v2) at (7,10) [aw] {$40v2$} ;

\draw (40v2) -- (31vp) ;
\draw (40v2) -- (2200vv) ;
\draw (40v2) -- (3100v2) ;
\draw (40v2) -- (22v2) ;

\draw (31vp) -- (22vp) ;
\draw (31vp) --(2110vp) ;
\draw (2200vv) --(2110vp) ;
\draw (2200vv) --(1120vv) ;
\draw (3100v2) -- (2110vp) ;
\draw (3100v2) -- (2020v2) ;
\draw (22v2) -- (1120vv) ;
\draw (22v2) -- (2020v2) ;
\draw (22v2) -- (13v2) ;
\draw (22v2) -- (t22v2) ;
\draw (d22v2) -- (22vp) ;

\draw (22vp) -- (13vp) ;
\draw (22vp) -- (2020vp) ;
\draw (22vp) -- (1120vp) ;
\draw (2110vp) -- (1111pp) ;
\draw (2110vp) -- (2020vp) ;
\draw (2110vp) -- (1120vp) ;
\draw (2110vp) -- (1021vp) ;
\draw (1120vv) -- (1120vp) ;
\draw (1120vv) -- (1021vp) ;
\draw (1120vv) -- (0022vv) ;
\draw (1120vv) -- (1030vv) ;
\draw (2020v2) -- (2020vp) ;
\draw (2020v2) -- (1021vp) ;
\draw (2020v2) -- (1030v2) ;
\draw (13v2) -- (1030vv) ;
\draw (13v2) -- (1030v2) ;
\draw (13v2) -- (04v2) ;
\draw (13v2) -- (t13v2) ;
\draw (d13v2) -- (13vp) ;

\draw (13vp) -- (04vp) ;
\draw (13vp) -- (1030vp) ;
\draw (1111pp) --(1021pp) ;
\draw (2020vp) --(1021pp) ;
\draw (1120vp) --(1021pp) ;
\draw (1120vp) -- (1030vp) ;
\draw (1120vp) -- (0022vp) ;
\draw (1021vp) -- (1021pp) ;
\draw (1021vp) -- (0022vp) ;
\draw (1021vp) -- (0031vp) ;
\draw (0022vv) -- (0022vp) ;
\draw (1030vv) -- (1030vp) ;
\draw (1030vv) -- (0031vp) ;
\draw (1030v2) -- (0031vp) ;
\draw (1030v2) -- (0040v2) ;
\draw (04v2) -- (0040v2) ;
\draw (04v2) -- (t04v2) ;
\draw (d04v2) -- (04vp) ;

\draw (1021pp) -- (0022pp) ;
\draw (0022vp) -- (0022pp) ;

\end{tikzpicture}

\caption{The degeneration scheme of biorthogonal systems}\label{figfulldeg}
 
\label{figdegscheme}
 
\end{figure} 
The graph should be read as follows. In the graph we placed all families of biorthogonal functions except for the $as$ ones (for those: see below). The height of a family denotes the dimension of the face which the family corresponds to, the top level ($40v2$) corresponds to a vertex, so dimension 0, while the bottom level ($0022pp$) corresponds to the interior of a $P_{III}$, which has dimension 5. This dimension is inversely proportional to the number of free parameters the limiting family still has. Indeed the top degeneration depends on $5-0=5$ free parameters (remember we have a balancing condition, which defines the sixth parameters in terms of the other five), while the bottom degeneration has no ($5-5=0$) parameters.

The edges between different families denote limit transitions. By the iterated limit property we know that we can take further limits of our biorthogonal families. One family is a limit of another one if one of the associated faces of the first family contains a face associated to the second family. This allows us to easily obtain the possible limit transitions.

We indicated the Pastro-polynomial type limits (i.e. the ones of form $pp$) with an oval. The rectangles actually contain much more subtle information. Underneath (i.e. one level lower) each of the limits denoted with a rectangle there exists a family with the same node, except that the $u$ parameters are of the $as$ type. Moreover there exists a limit from the family in the box to the corresponding $as$-type limit, and this is the only limit from a non-$as$ type family. The limits between different $as$-type families are given by the edges between the corresponding families with a rectangle around them. \footnote{This somewhat unconventional way of representing things was performed because the picture would otherwise become much too cluttered.}

\end{document}